%% file: affine.tex
\documentclass[12pt]{amsart}

\usepackage{amsmath}
\usepackage{verbatim}
\usepackage{amssymb}
\usepackage{amsthm}
\usepackage{pstricks}
\usepackage{pst-node}
\usepackage{pst-plot}
\usepackage{subfigure}
\usepackage{graphicx}
\usepackage{pdfpages}
\usepackage{url}

\numberwithin{equation}{section}

\newtheorem{thm}{Theorem}[section]
\newtheorem{prop}[thm]{Proposition}
\newtheorem{conj}[thm]{Conjecture}
\newtheorem{question}[thm]{Question}
\newtheorem{lem}[thm]{Lemma}
\newtheorem{cor}[thm]{Corollary}

\theoremstyle{definition}

\newtheorem{rmk}[thm]{Remark}
\newtheorem{ex}[thm]{Example}

\def\s{{\mathfrak s}}
\def\W{{\mathcal W}}

\def\y{{\bf y}}

\def\w{{\bf w}}
\def\z{{\bf z}}

\def\Z{{\mathbb Z}}
\def\aW{\hat{S}_n}

\begin{document}
\title{Discrete solitons in infinite reduced words}
\author{Max Glick \and Pavlo Pylyavskyy}
\address{Department of Mathematics, University of Minnesota,
Minneapolis, MN 55455, USA}
\thanks{M. G. was partially supported by NSF grant DMS-1303482.  P. P. was partially supported by NSF grants DMS-1148634, DMS-1351590, and Sloan Fellowship.}

\begin{abstract}
We consider a discrete dynamical system where the roles of the states and the carrier are played by translations in an affine Weyl group of type $A$. The Coxeter generators are enriched by parameters, and the interactions with the 
carrier are realized using Lusztig's braid move $(a,b,c) \mapsto (bc/(a+c), a+c, ab/(a+c))$. We use wiring diagrams on a cylinder to interpret chamber variables as $\tau$-functions.  This allows us to realize our systems as reductions of the Hirota bilinear difference equation and thus obtain $N$-soliton solutions.  
\end{abstract}

\maketitle

\setcounter{tocdepth}{1}
\tableofcontents

\section{Introduction}

\subsection{Solitons}

One of the most remarkable properties a non-linear differential equation can possess is the existence of {\it {soliton solutions}}. The simplest and one of the first such equation to be discovered was the Korteweg–-de Vries equation, or KdV:
$$u_t + 6 u u_x + u_{xxx} = 0.$$
Here $u = u(x,t)$ is a function of two continuous parameters $x$ and $t$, and the lower indices denote derivatives with respect to the specified variables. The three universal features characterizing soliton solutions are as follows.
\begin{enumerate}
 \item Existence of {\bf {$1$-soliton solutions}}. These are single hump solutions which propagate with time without changing their shape. Such solutions were first described by Korteweg and de Vries \cite{KdV}.
 \item Existence of {\bf {multi-soliton solutions}}. These behave in a manner resembling linear combinations of $1$-soliton solutions. They are not however merely linear combinations, as one cannot add solutions of non-linear equations. Experimentally, one observes that several solitons within a multi-soliton solution interact with each other, regaining their original shapes once the interaction is over. The existence of such pseudo linear combination solutions for non-linear equations is a miracle, closely related to the {\it {integrability}}
 property of the equations. Multi-soliton solutions were first observed experimentally
 by Kruskal and Zabusky \cite{KZ}, and the exact formulas for such solutions were found by Gardner, Greene, Kruskal and Miura \cite{GGKM}. The term soliton was coined by Kruskal and Zabusky to emphasize the particle-like nature of those waves. 
 \item {\bf {Spontaneous emergence}} of solitons, or {\bf {soliton resolution}}. If one starts with arbitrary initial conditions $u(x,0)$ that decay rapidly as $x \rightarrow \pm \infty$, one can try to solve for $u(x,t)$ for $t>0$, either numerically or exactly. What one observes is the emergence of several soliton humps moving in one direction, and of chaotic looking {\it {radiation}} moving in the other direction (see Figure \ref{figEmergence} for a typical picture). The exact number and sizes of solitons emerging can be found by solving the {\it {Shr\"odinger scattering problem}}, also known as the {\it {direct scattering problem}}, as discovered by Gardner, Greene, Kruskal and Miura \cite{GGKM}.
\end{enumerate}

To illustrate, 
Figure \ref{figSoliton} shows a $2$-soliton. The first hump, which is both larger and faster, passes through the second hump. Both retain their original shape afterwords.
\begin{figure} [h!]
\includegraphics[width=4cm,height=4cm]{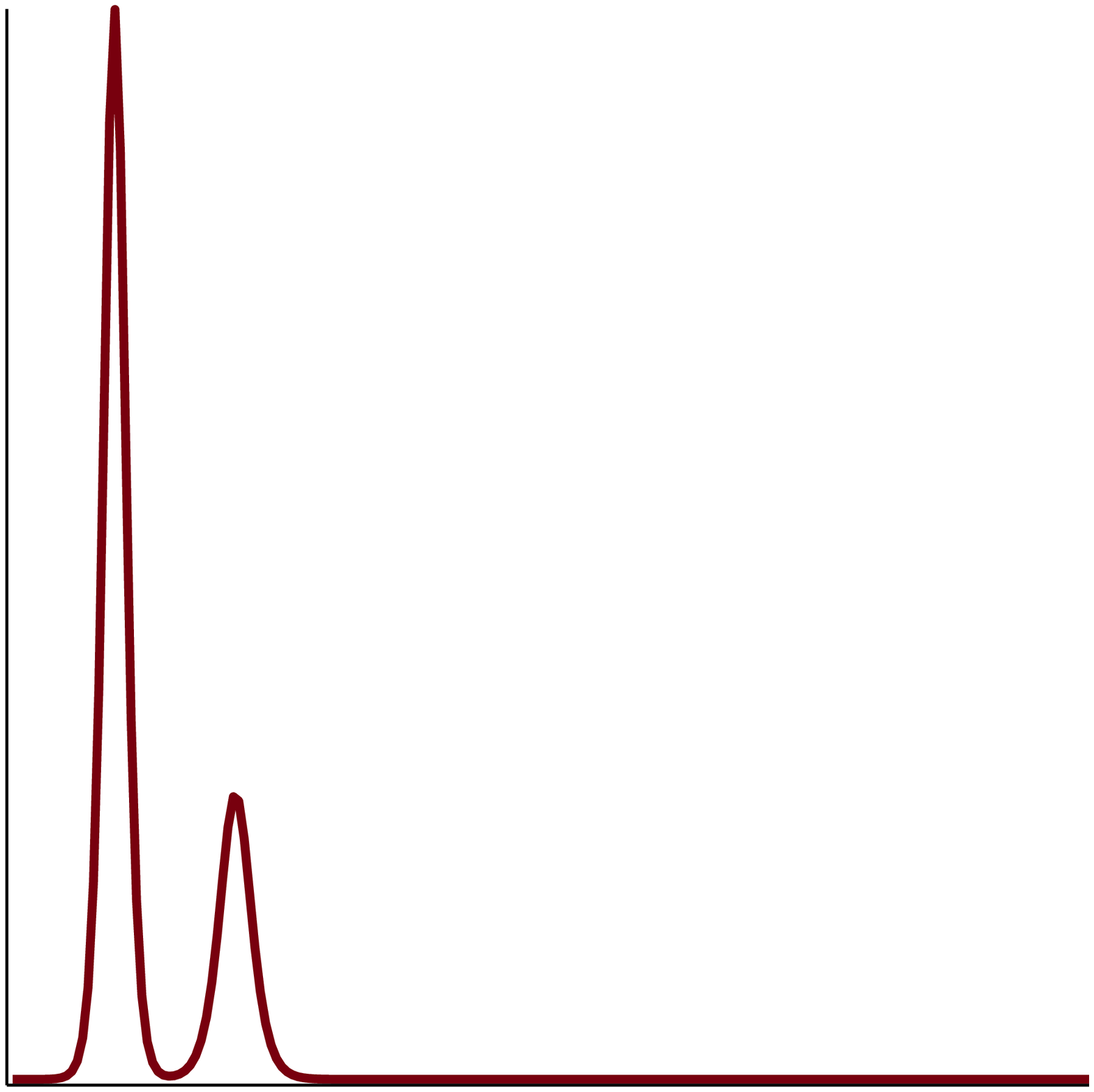}
\includegraphics[width=4cm,height=4cm]{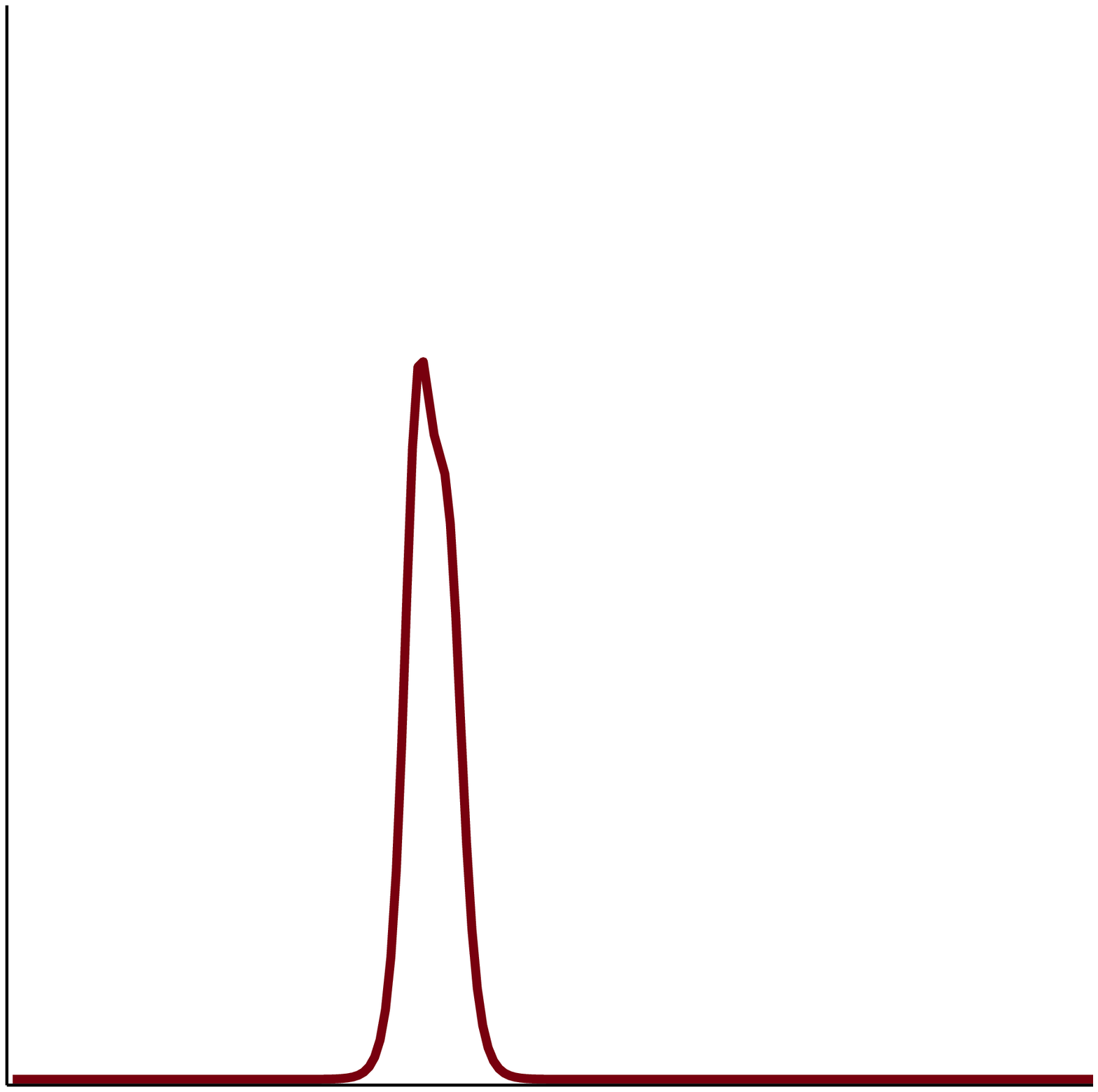}
\includegraphics[width=4cm,height=4cm]{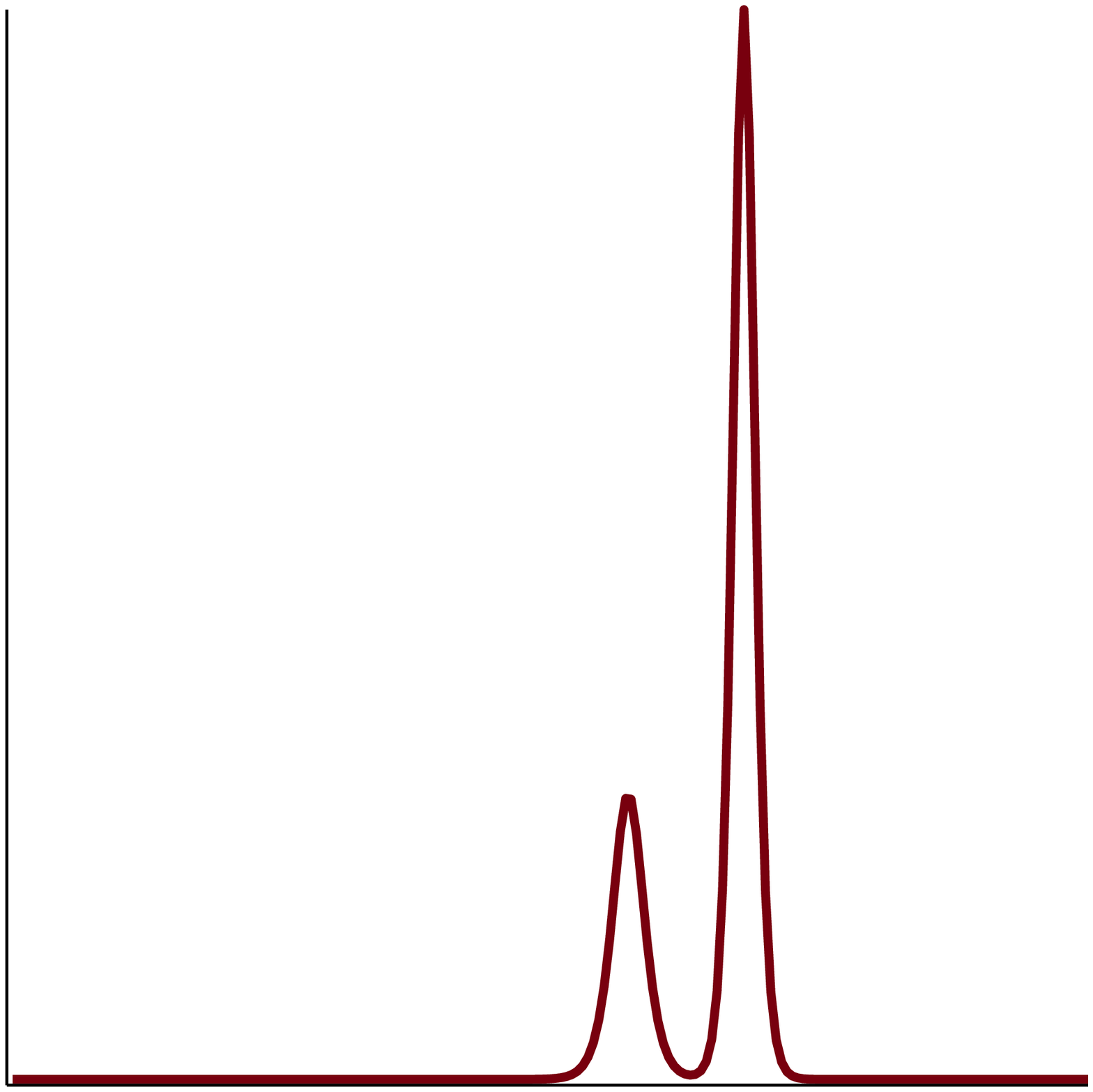}
\caption{A 2-soliton}
\label{figSoliton}
\end{figure}
The two humps do interact however, as can be detected by measuring the difference between their positions and their would be positions if they were just moving with constant velocity. 

We refer the reader to \cite{K} for an accessible introduction to solitons, and to \cite{AC, DJ} for a more comprehensive treatment of the subject. 

\subsection{Three levels of discreteness}

KdV equation and its solutions are an example of a {\it {continuous system}}. Remarkably, there are also {\it {discrete}} and {\it {ultradiscrete}} systems exhibiting solitonic behavior. The main examples of those are Hirota's {\it {discrete KdV}} \cite{H} and 
Takahashi-Satsuma {\it {box ball system}} \cite{TS}. The differences between the three levels of discreteness are summarized in the following table. 

\begin{center}
\begin{tabular}{|l| c| c| c|}
\hline
$ $ & Continuous (e.g. KdV) & Discrete (e.g. dKdV) & Ultradiscrete (e.g. box ball) \\[.5ex]
\hline
Space & continuous & discrete & discrete \\[.5ex]
\hline
Time & continuous & discrete & discrete \\[.5ex]
\hline
Range & continuous & continuous & discrete \\[.5ex]
\hline
Soliton & \scalebox{0.5}{\input{sol21.pstex_t}} & \scalebox{0.5}{\input{sol22.pstex_t}} & \scalebox{0.5}{\input{sol23.pstex_t}}\\[.5ex]
\hline 
Carrier  & no & yes & yes \\[.5ex]
\hline
\end{tabular}
\end{center}

In the case of the dKdV and the box ball system, the evolution can be described by passing a {\it {carrier}} through an infinite sequence of states. The interaction between the carrier and a state is realized by $\mathfrak{sl}_2$ versions of {\it {geometric}} and 
{\it {combinatorial $R$-matrices}}. 

\subsection{Evolution via Lusztig relations in loop groups}

In this paper we are going to consider a different model for the carrier-state interaction. Specifically, we are going to consider pairs of translations in the weight lattice of the affine Coxeter group $\hat S_n$. Assume $u$ and $v$ are the pair of elements of $\hat S_n$ which realize those translations. We call such elements {\it {glides}}, see the next section for details. 

Let us fix reduced decompositions $u = s_{i_1} \cdots s_{i_l}$ and $v = s_{j_1} \cdots s_{j_m}$. We are going to assume that $vu$ is reduced. Let us enrich each Coxeter generator $s_i$ by a real parameter $a$,  denoting the result $s_i(a)$. We call such parameters {\it {wall parameters}},
for reasons to be clear later. 
The $s_i(a)$ can be now thought of as (exponents of) the {\it {Chevalley generators}} of the polynomial loop into $GL_n$, 
$$s_i(a) = \exp(a \epsilon_i),$$ where $\epsilon_i$ is an upper triangular Chevalley generator of the corresponding Lie algebra. The study of the {\it {totally positive parts}} of algebraic groups generated by exponents $\exp(a \epsilon_i)$ with $a>0$ was initiated by Lusztig \cite{L}.
The study of the totally positive parts of loops into $GL_n$ was undertaken in \cite{LP1,LP2}. We refer the reader specifically to \cite{LP2} for the details. 

Now that our glides are enriched by parameters, we can consider their interaction, consisting of commuting the two translations with each other. The non-trivial part is to say what happens to the parameters. This is uniquely determined however by requesting that the identity 
\begin{align*}
u'v'&= s_{i_1+k_2}(y_1') \cdots s_{i_l+k_2}(y_l')s_{j_1-k_1}(z_1') \cdots s_{j_m-k_1}(z_m') \\
&= s_{j_1}(z_1) \cdots s_{j_m}(z_m)s_{i_1}(y_1) \cdots s_{i_l}(y_l) = vu
\end{align*}
holds in the loop group, where $k_1$ and $k_2$ are certain integer offsets needed to make the identity work (see Section \ref{secSetup}).  More concretely, the $y_i'$ and $z_i'$ are computed recursively from $y_1,\ldots, y_l$ and $z_1,\ldots, z_m$ by the Lusztig relations
\begin{equation} \label{eqLusztig}
\begin{split}
&s_i(a) s_j(b) s_i(c) = s_j \left(\frac{bc}{a+c} \right) s_i(a+c) s_j \left(\frac{ab}{a+c} \right), \text{ if } i-j \equiv \pm 1 \pmod{n}; \\
&s_i(a) s_j(b) = s_j(b) s_i(a), \quad \text{ if } i-j \not\equiv 0,\pm 1 \pmod n.
\end{split}
\end{equation}

Fix $u$ (called the \emph{state glide}) and $v$ (called the \emph{carrier glide}) as above.  Also fix two sets of parameters, one for $u$ and one for $v$, which have the property of remaining unchanged when the carrier interacts with the state. We refer to those choices of parameters as {\underline {the vacuum}} and {\underline {the initial carrier}} respectively. From this data we define a system, \emph{affine dKdV}, whose time evolution is carried out as follows.  
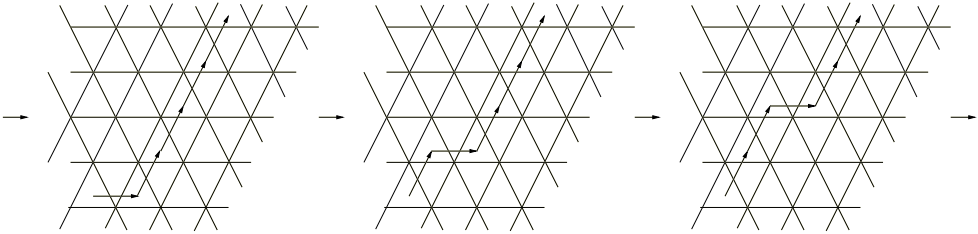
\begin{figure}[ht]
\scalebox{0.6}{\input{aff4.pstex_t}}
    \caption{The carrier interacting with the states}
    \label{fig:aff4}
\end{figure}
The initial data consists of an infinite power of the state glide, with the parameters approaching the vacuum at the limit in both directions.  Then we {\underline {push the carrier through each state}}, i.e. have it interact with each state in turn, assuming the carrier approaches from one direction with parameters 
equal to the initial carrier in the limit. This constitutes one step of the time evolution.

\subsection{Solitons: existence and spontaneous emergence}

Assume we choose arbitrarily the following data. First, we choose $n$ and two glides $v$ and $u$, with a fixed choice of reduced words $u = s_{i_1} \cdots s_{i_l}$ and $v = s_{j_1} \cdots s_{j_m}$. Assume that the concatenation $s_{j_1} \cdots s_{j_m}s_{i_1} \cdots s_{i_l}$
is also reduced. Assume we choose positive parameters in the corresponding initial carrier $\z_{-\infty}$ and vacuum $\w$ so that the parameters remain the same after the interaction. 
Finally, assume we start with the initial state which is equal to vacuum everywhere but in a finite number of states, and we run repeatedly
the time evolution. 

We are going to fix $1 \leq h \leq l$ and consider the function $f_h: \mathbb Z \rightarrow \mathbb R_{>0}$ defined as $$f_h(j) = \text{ the parameter of } s_{i_h} \text{ in the } j \text{-th state}.$$ 
In other words, in each state we pick the $h$-th parameter from the beginning. As time evolution proceeds, the values of $f_h$ change. One can observe then that several ``waves'' which preserve their shape and move with velocity approaching constant emerge, leaving a certain amount of 
chaotic ``radiation'' behind. The definition is imprecise since we
only ``see'' our values at integer points $j \in \mathbb Z$, and thus fluctuations occur that prevent the shape of being exactly preserved. This is the usual situation for discrete solitons. We  make the following conjecture. 

\begin{conj}
 For any choice of the data as above spontaneous emergence of solitons occurs for a generic choice of the initial data. 
\end{conj}

\begin{ex}
 Let $n=3$ and pick glides $u=s_2s_1$, $v=s_1s_2s_1s_0$.  One can check the vacuum $\w = (1,4)$ and initial carrier $\z_{-\infty} = (3,4,1,4)$ are unchanged by interaction.  We let $\ldots, \y_1,\y_2,\ldots$ be the initial states with $\y_j = \w$ except for $j=51,52,\ldots, 100$.  In this range, assign independently a random value to both entries of $\y_j$ with said values drawn from the interval $[1,10)$.  Figure \ref{figEmergence} shows one possible outcome after $50$ steps of the time involution.  In the notation above, $h=1$ meaning we plot the first entry of each state.  The solitons that emerge can be seen to the right in the second plot.
\end{ex}

\begin{figure} [h!]
\includegraphics[width=4cm,height=4cm]{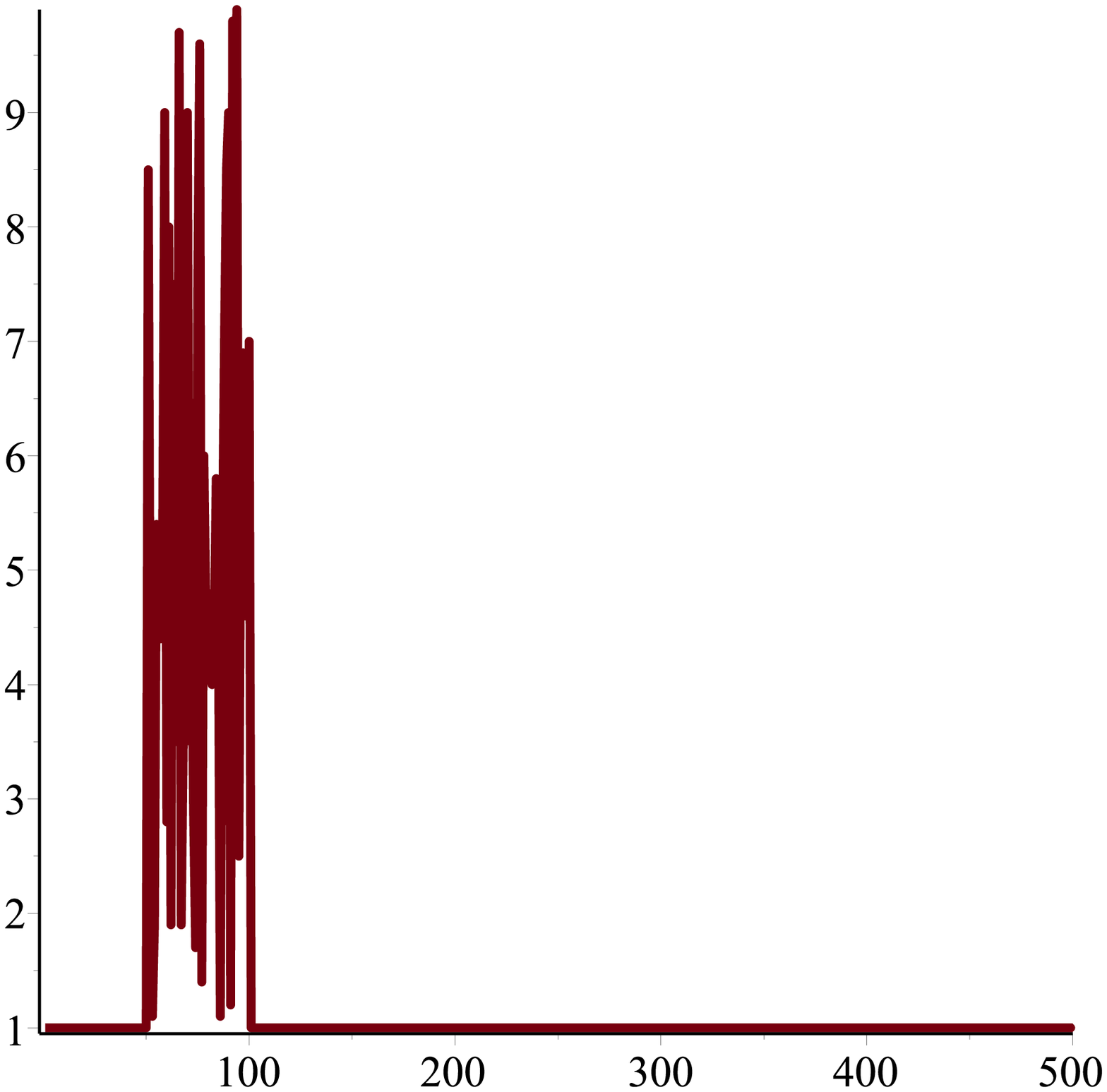}
\includegraphics[width=4cm,height=4cm]{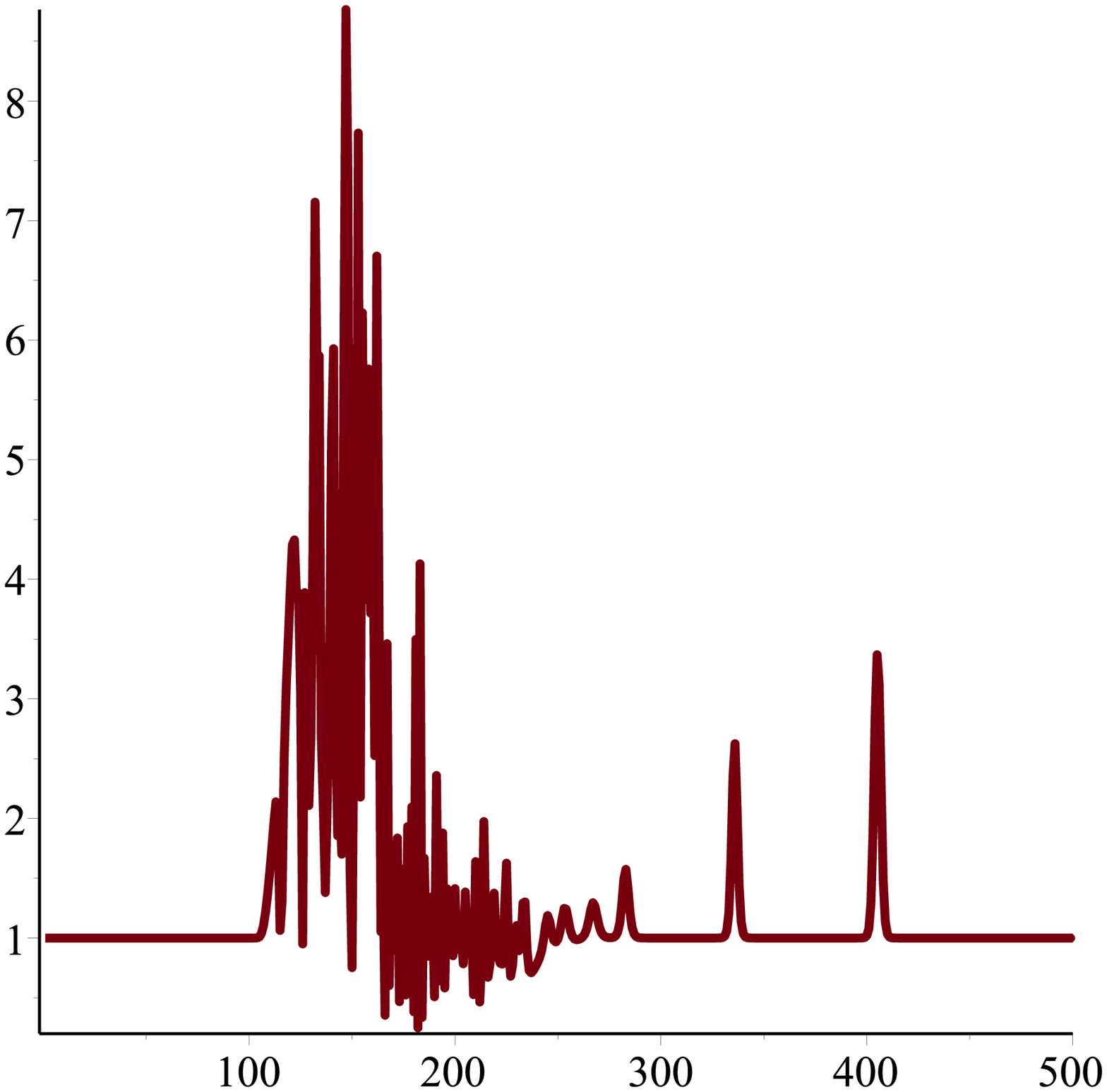}
\caption{Initial data for affine dKdV (left) and the resulting states after $50$ steps (right).  The latter illustrates direct scattering with solitons in front followed by radiation.}
\label{figEmergence}
\end{figure}

Now, a weaker goal than proving this conjecture would be to prove existence of $N$-soliton solutions, i.e. states that have $N$ soliton humps and no radiation. 
We state the main theorem of this paper.

\begin{thm} \label{thmMain}
 For certain natural choices of the parameters of the model, $N$-soliton solutions exist. 
\end{thm}

Figure \ref{figInteraction1} depicts a $2$-soliton solution of an instance of affine dKdV.  The larger and smaller soliton waves pass through each other during roughly the time interval $t=10$ to $t=30$.

\begin{figure}
\vspace{-0.5in}
\includegraphics[width=12cm,height=12cm]{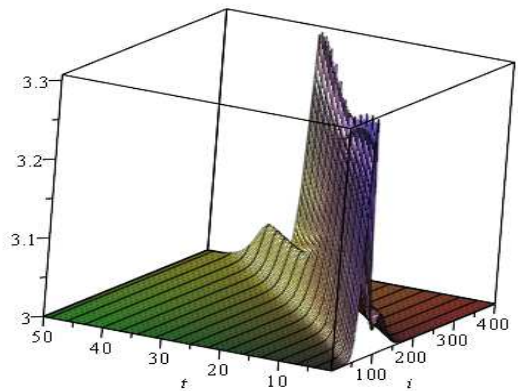}
\vspace{-1.7in}
\caption{A 2-soliton solution to affine dKdV}
\label{figInteraction1}
\end{figure}

We prove Theorem \ref{thmMain} by finding a lift from affine dKdV to the Hirota bilinear difference equation.  The latter is known to possess soliton solutions, see e.g. \cite{N} for explicit formulas.  From these formulas, the solitons for affine dKdV are obtained by taking double ratios in a pattern that depends on the glides $u$ and $v$.  The conversion between the two systems, which is itself a main focus of the current paper, touches on interesting combinatorics of wiring diagrams, weighted directed networks, and the alcove model for the affine symmetric group.

The remainder of this paper is organized as follows.  In Section \ref{secSetup} we formally define affine dKdV in terms of dynamics of infinite reduced words in the affine symmetric group.  Sections \ref{secNetworkModel} and \ref{secNetwork} provide a model for our system in terms of dynamics, previously studied in \cite{LP3}, on directed networks.  A quick application of the network model is that it gives rise to a method for producing commuting state and carrier pairs.  We interpret face variables as $\tau$-functions and describe their dynamics in Section \ref{secBHZ}.  This all culminates in the formulas for solitons given in Section \ref{sec:nsol} as promised in Theorem \ref{thmMain}.  Section \ref{secAlcove} introduces the alcove model which clarifies several matters including how to find compatible glides and how to arrange for all parameters to be positive.  In Section \ref{secTopological} we explain an unexpected feature of our system, that the parameter space of one-solitons is disconnected, and we explore how soliton speed varies from component to component.  Lastly, Section \ref{sec:cf} proves that in many cases affine dKdV admits a local definition, i.e. one not depending on a dynamic carrier.

\section{Dynamics in infinite reduced words} \label{secSetup}

In this section, we introduce the affine dKdV system.  We begin by describing the fixed parameters of the system which consist of a combination of combinatorial data (a pair of reduced words in an affine symmetric group) and numerical data. 

\subsection{Carrier description via Yang-Baxter moves}
The affine symmetric group $\hat{S}_n$ is defined by generators $s_0, s_1, \ldots, s_{n-1}$ (with indices always taken modulo $n$) and relations
\begin{displaymath}
\begin{split}
&s_i^2 = 1 \text { for all } i; \\
&s_is_js_i = s_js_is_j \text{ if } i-j \equiv \pm 1 \pmod{n}; \\
&s_is_j = s_js_i \text{ if } i-j \not\equiv 0,\pm 1 \pmod n.
\end{split}
\end{displaymath}
The relations are clearly preserved by cyclic shifts, so there is an automorphism $\rho : \hat{S}_n \to \hat{S}_n$ such that $\rho(s_i) = s_{i+1}$ for all $i$.

  There is a homomorphism $\phi: \hat{S}_n \to S_n$ to the ordinary symmetric group that takes each $s_i$ to a transposition in $S_n$.  Specifically $\phi(s_i) = (i\ \ i+1)$ for $i=1,2,\ldots, n-1$ and $\phi(s_0) = (1\ \ n)$.  A $\emph{translation}$ in $\hat{S}_n$ is an element that maps to the identity in $S_n$.  The set of translations form an Abelian subgroup of the affine symmetric group.  

Say that $g \in \hat{S}_n$ is a \emph{glide} if 
\begin{displaymath}
\phi(g) = \left(
\begin{array}{lllllll} 
1 & 2 & \cdots & n-k & n-k+1 & \cdots & n \\
1+k & 2+k & \cdots & n & 1 & \cdots & k \\
\end{array}
\right)
\end{displaymath}
for some $k=0,1,\ldots, n-1$.  Call $k$ the \emph{offset} of the glide.  
It is not hard to see that $\rho$ sends glides to glides and preserves offsets.

\begin{lem} \label{lemRho}
Let $g_1, g_2 \in \hat{S}_n$ be glides with offsets $k_1$ and $k_2$ respectively.  Then
\begin{displaymath}
g_2g_1 = \rho^{k_2}(g_1)\rho^{-k_1}(g_2).
\end{displaymath}
\end{lem}

\begin{proof}
It is easier to prove this result in greater generality, namely in the extended affine symmetric group defined by adding to $\hat{S}_n$ a new generator $\tau$ together with relations $\tau s_i \tau^{-1} = s_{i+1}$ for all $i$.  The map to the symmetric group $\phi$ and the automorphism $\rho$ both extend to this setting via 
\begin{align*}
\phi(\tau) &= (1\ 2\ \cdots \ n) \\
\rho(\tau) &= \tau
\end{align*}
The notions of translation, glide, and offset are all defined the same way as before and the set of translations still form an Abelian subgroup.  Note that $\tau$ is itself a glide of offset $1$.  The upshot is that $\rho$ is now an inner automorphism given by conjugation by $\tau$.

It is clear that the product of two glides is a glide and that the offset of the product is the sum of the offsets.  Therefore $\tau^{-k_2}g_2$ and $g_1\tau^{-k_1}$ are translations and hence commute with each other:
\begin{align*}
\tau^{-k_2}g_2g_1\tau^{-k_1} &= g_1\tau^{-k_1-k_2}g_2 \\
g_2g_1 &= (\tau^{k_2}g_1\tau^{-k_2})(\tau^{-k_1}g_2\tau^{k_1}) \\
&= \rho^{k_2}(g_1)\rho^{-k_1}(g_2)
\end{align*}
as desired.
\end{proof}

Let $u,v \in \hat{S}_n$ be glides.  Fix reduced decompositions $u = s_{i_1} \cdots s_{i_l}$ and $v = s_{j_1} \cdots s_{j_m}$.  Assume that  
\begin{displaymath}
vu = s_{j_1} \cdots s_{j_m}s_{i_1} \cdots s_{i_l}
\end{displaymath}
is reduced, i.e. that $vu$ has Coxeter length $m+l$.  Then by Lemma \ref{lemRho} another decomposition of the same element is
\begin{displaymath}
vu = \rho^{k_2}(u)\rho^{-k_1}(v) = s_{i_1+k_2} \cdots s_{i_l+k_2}s_{j_1-k_1} \cdots s_{j_m-k_1},
\end{displaymath}
so there is a sequence of braid moves and commutation relations carrying one word to the other.  Introducing weights $\y = (y_1,\ldots, y_l)$ and $\z = (z_1,\ldots, z_m)$ the rules \eqref{eqLusztig} can be used to obtain
\begin{displaymath}
(\y',\z') = F_{v,u}(\z,\y)
\end{displaymath}
such that 
\begin{align*}
&s_{i_1+k_2}(y_1') \cdots s_{i_l+k_2}(y_l')s_{j_1-k_1}(z_1') \cdots s_{j_m-k_1}(z_m') \\
&= s_{j_1}(z_1) \cdots s_{j_m}(z_m)s_{i_1}(y_1) \cdots s_{i_l}(y_l)
\end{align*}

Note that the function $F = F_{v,u}: (\mathbb{R}_{>0})^{m+k} \to (\mathbb{R}_{>0})^{k+m}$ depends on $u$ and $v$ and in fact on the choice of their reduced decompositions.  It does not however depend on the choice of elementary moves used to go from $vu$ to $\rho^{k_2}(u)\rho^{-k_1}(v)$ as follows from results of \cite{L}.

\begin{ex} \label{exF}
Let $n=3$ and let $u=s_1s_2s_1s_0$, $v=s_1s_0$.  Then $u$ and $v$ are glides with offsets $0$ and $2$ respectively.  Therefore 
\begin{displaymath}
(s_1s_0)(s_1s_2s_1s_0) = vu = \rho^2(u)v = (s_0s_1s_0s_2)(s_1s_0).
\end{displaymath}
In this case, the equality follows from a single $s_1s_0s_1=s_0s_1s_0$ braid move.  Adding weights, 
\begin{displaymath}
(s_1(e)s_0(f))(s_1(a)s_2(b)s_1(c)s_0(d)) = (s_0(a')s_1(b')s_0(c')s_2(d'))(s_1(e')s_0(f'))
\end{displaymath}
where $a' = fa/(e+a)$, $b'=e+a$, $c' = ef/(e+a)$, $d'=b$, $e'=c$, and $f'=d$.  The map $F$ is
\begin{displaymath}
F: ((e,f),(a,b,c,d)) \mapsto ((a',b',c',d'),(e',f')).
\end{displaymath}
\end{ex}

Fix glides with reduced decompositions $u = s_{i_1} \cdots s_{i_l}$ and $v = s_{j_1} \cdots s_{j_m}$ and assume as before that $vu$ is itself reduced.  Call $s_{i_1} \cdots s_{i_l}$ the \emph{state word} and $s_{j_1} \cdots s_{j_m}$ the \emph{carrier word}.  A \emph{state} is an element of $(\mathbb{R}_{>0})^l$, thought of as a choice of weights for $u$.  A \emph{carrier} is an element of $(\mathbb{R}_{>0})^m$, thought of as choice of weights for $v$.  The operation $(\z,\y) \mapsto (\y',\z') = F(\z,\y)$ is called an \emph{interaction} of the carrier $\z$ with the state $\y$.

In addition to $u$ and $v$, fix a pair consisting of a state $\w$ and a carrier $\z_{-\infty}$ such that $F(\z_{-\infty},\w) = (\w,\z_{-\infty})$.  Here $\w$ is called the \emph{vacuum} and $\z_{-\infty}$ is called the \emph{initial carrier}.  Table \ref{tabForm1} summarizes this setup.

\begin{table}
\begin{tabular}{|l|l|l|}
\hline
Name & Notation & Properties \\
\hline
State word & $u = s_{i_1} \cdots s_{i_l} \in \hat{S_n}$ & $u$ a glide\\
Carrier word & $v = s_{j_1} \cdots s_{j_m} \in \hat{S_n}$ & $v$ a glide, $vu$ reduced \\
Vacuum & $\w \in (\mathbb{R}_{>0})^l$ & \\
Initial Carrier & $\z_{-\infty} \in (\mathbb{R}_{>0})^m$ & $F_{v,u}(\z_{-\infty},\w) = (\w,\z_{-\infty})$\\
\hline
\end{tabular}
\caption{The parameters for the first formulation of affine dKdV}
\label{tabForm1}
\end{table}

Affine dKdV is a discrete dynamical system operating on the space of infinite state sequences $\ldots, \y_0, \y_1, \y_2, \ldots$ satisfying
\begin{displaymath}
\lim_{i\to \pm \infty} \y_i = \w.
\end{displaymath}
One step of time evolution consists of an infinite sequence of interactions in which a carrier moves left to right passing through each state one by one.  The carrier is initialized to $\z_{-\infty}$ and the interactions are
\begin{displaymath}
(\z_i, \y_i) \mapsto (\y_i', \z_{i+1}) = F(\z_i, \y_i)
\end{displaymath}
as $i$ ranges over $\mathbb{Z}$.  The output is the new state sequence $\ldots, \y_0', \y_1', \y_2', \ldots$.


Now, the map $F$ can be considered to be a weighted version of the identity $vu = \rho^{k_2}(u)\rho^{-k_1}(v)$.  In the same spirit, the full system as just defined should be thought of as dynamics on infinite reduced words as follows.  Consider the infinite word 
\begin{displaymath}
\cdots \rho^{3k_1}(u) \rho^{2k_1}(u) \rho^{k_1}(u) u \rho^{-k_1}(u)     \cdots
\end{displaymath}
and imagine inserting a rotation of $v$ very far to the left.  More precisely, if $\rho^{jk_1}(v)$ is immediately to the left of $\rho^{jk_1}(u)$ then we can push it to the right using
\begin{displaymath}
\rho^{jk_1}(v)\rho^{jk_1}(u) = \rho^{jk_1+k_2}(u)\rho^{jk_1-k_1}(v).
\end{displaymath}
Because $\rho^{jk_1}$ is an automorphism, the corresponding transformation of weights is still $F$.  Once this step is done we have that $\rho^{(j-1)k_1}(v)$ is to the left of $\rho^{(j-1)k_1}(u)$, so the process of pushing $v$ right can continue.  Letting the $\y_i$ for $i \in \mathbb{Z}$ be the weights of the $\rho^{-ik_1}(u)$ before the sweep, we obtain the $\y_i'$ as the weights of the $\rho^{-ik_1+k_2}(u)$ afterwords.

\subsection{Commuting pairs}
In Section \ref{secNetwork} we shall explain {\it {a way}} to create commuting carrier and vacuum pairs. For now, we shall make the following observation. Call a glide $v'$ {\it {primitive}} if for no other glide $v$ with offset $k$ 
and no integer $\ell \geq 1$ we have 
$$v' = v \rho^{-k}(v) \dotsc \rho^{-\ell k}(v).$$
Assume $\z = (z_1, \ldots, z_{m(\ell+1)})$ and $\w = (w_1, \ldots, w_{m'(\ell'+1)})$ are two states with the underlying glides $v \rho^{-k}(v) \dotsc \rho^{-\ell k}(v)$ and $u \rho^{-k'}(u) \dotsc \rho^{- \ell' k'}(u)$, where $v$ and $u$ are distinct primitive glides with offsets $k$ and $k'$ respectively. 
\begin{conj}
 If $F(\z, \w) = (\w, \z)$ and all parameters $z_i, w_j$ are positive, then $z_i = z_{i+m}$ and $w_j = w_{j+m'}$ for all $i$ and $j$, and $$F((z_1, \ldots, z_m), (w_1, \ldots, w_{m'})) = ((w_1, \ldots, w_{m'}), (z_1, \ldots, z_m)).$$
\end{conj}
In simple terms, the conjecture says that the commuting states corresponding to powers of primitive glides come from commuting states for the individual primitive glides.

\begin{ex}
 Let $n=3$, $v = s_1 s_2$, and $u = s_2 s_1$. Then $m=m'=2$, $k=1$ and $k'=2$. Take $\ell = \ell' = 1$, $\z = (a,b,c,d)$, $\w = (e,f,g,h)$. We have 
 $$[(s_1(a) s_2(b)) (s_0(c) s_1(d))] [(s_2(e) s_1(f)) (s_0(g) s_2(h))] =$$ $$= [(s_1(e') s_0(f')) (s_2(g') s_1(h'))] [(s_0(a') s_1(b')) (s_2(c') s_0(d'))].$$
 If we now assume that $a=a', \ldots, h=h'$, and that all parameters are positive, it is not hard to check that the only solution is as follows:
  $$[(s_1(d+f) s_2(d)) (s_0(d+f) s_1(d))] [(s_2(d+f) s_1(f)) (s_0(d+f) s_2(f))] =$$ $$= [(s_1(d+f) s_0(f)) (s_2(d+f) s_1(f))] [(s_0(d+f) s_1(d)) (s_2(d+f) s_0(d))].$$
  Thus, we see that $a=c$, $b=d$, $e=g$ and $f=h$, as claimed by the conjecture. Furthermore, the commutation for $\z$ and $\w$ comes from the commutation relations for individual primitive glides:
  $$F((d+f,d),(d+f,f)) = ((d+f,f),(d+f,d)).$$
\end{ex}

\begin{ex}
 Note that if the positivity condition is dropped, the claim of the conjecture becomes false. For example, in the previous example an alternative choice of parameters is as follows:
 $$[(s_1(-d-f) s_2(d)) (s_0(-d+f) s_1(d))] [(s_2(-d-f) s_1(f)) (s_0(d-f) s_2(f))] =$$ $$= [(s_1(-d-f) s_0(f)) (s_2(d-f) s_1(f))] [(s_0(-d-f) s_1(d)) (s_2(-d+f) s_0(d))].$$
 We see that $a \not = c$ and $e \not = g$.
\end{ex}

\begin{question}
 For a fixed choice of $\w$, meaning both the choice of a glide $u$ and of positive parameters, what choices of another glide $v$ can produce $\z$ such that $F(\z,\w)=(\w,\z)$?
\end{question}

\begin{conj}
 Assume we have made a choice of glides $u,v$ and of positive parameters for $\w$, but not of the parameters of $\z$. Then if there exists a choice of $\z$ such that $F(\z,\w)=(\w,\z)$, this choice is unique. 
\end{conj}

\section{Network model on a cylinder} \label{secNetworkModel}

The following adaptation of the model introduced in \cite{LP3} will be extremely useful for our purposes. 
Consider an infinite cylinder with $n$ horizontal wires which may cross to form an infinite wiring diagram.

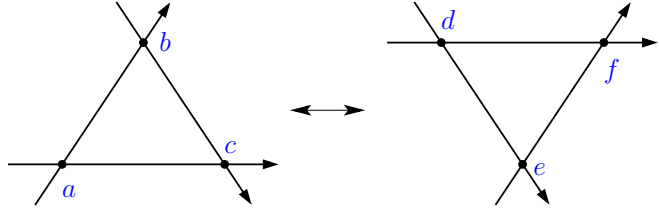
\begin{figure}[ht]
\scalebox{0.6}{\input{sol2.pstex_t}}
    \caption{The Yang-Baxter move}
    \label{fig:sol2}
\end{figure}

Consider a local {\it {Yang-Baxter move}} we do on such a network as shown in Figure \ref{fig:sol2}.  Here we have {\it {Lusztig's relations}}
$$d = {bc}/{(a+c)}, e = a+c, f = {ab}/{(a+c)},$$ $$a = ef/(d+f), b = d+f, c = de/(d+f)$$

Now, assume each wire has a parameter associated to it. Assign a variable to each {\it {chamber}}, which means a face cut out by wires. We shall refer to those variables as {\it {chamber variables}}, or {\it {$\tau$-functions}}. 
\begin{figure}[ht]
\scalebox{0.6}{\input{sol3.pstex_t}}
    \caption{The transition between the vertex and the chamber variables}
    \label{fig:sol3}
\end{figure}
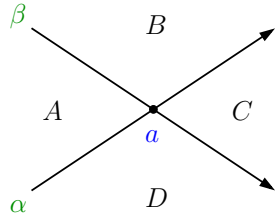
Assume we have assigned the chamber variables so that 
at any crossing of two wires the variable $a$ at the crossing is given by $$a = (\beta - \alpha) \frac{BD}{AC},$$ where $\alpha$ and $\beta$ are the wire parameters and $A,B,C,D$ are the chamber variables, as shown in Figure \ref{fig:sol3}. 
Note that the orientation of the wires matters for determining the factors in this formula. 

\begin{figure}[ht]
\scalebox{0.6}{\input{sol4.pstex_t}}
    \caption{The enriched Yang-Baxter move}
    \label{fig:sol4}
\end{figure}
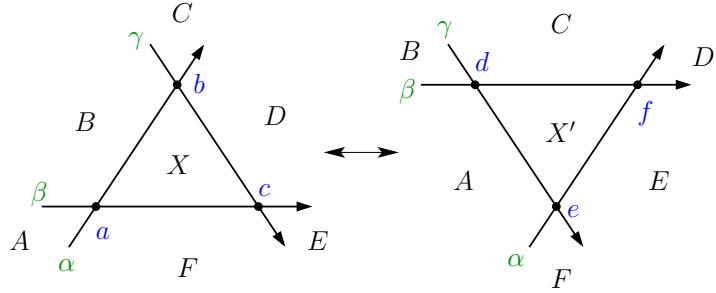
Consider an enriched Yang-Baxter move, which evolves chamber variables by leaving all of them but one the same, and changing the mutating variable $X$ according to the formula 
$$(\gamma-\alpha)XX' = (\beta-\alpha)BE+(\gamma-\beta)AD,$$ with the parameters assigned as in Figure \ref{fig:sol4}. 
\begin{lem} \label{lem:YB}
 The enriched Yang-Baxter move induces the usual Yang-Baxter move on vertex variables. 
\end{lem}

\begin{proof}
 Let us verify the formula for the new vertex variable $d$, variables $e$ and $f$ can be treated in the same manner. We have 
 $$d = (\gamma-\beta)\frac{AC}{BX'} = \frac{(\gamma-\alpha)(\gamma-\beta)ACX}{B((\beta-\alpha)BE + (\gamma-\beta)AD)} = \frac{(\gamma-\alpha)\frac{CX}{BD}(\gamma-\beta)\frac{DF}{EX}}{(\beta-\alpha)\frac{FB}{AX} + (\gamma-\beta)\frac{FD}{EX}} = \frac{bc}{a+c}.$$
\end{proof}

\section{Dynamics in networks} \label{secNetwork}
An infinite wiring diagram with $n$ wires on a cylinder is a pictorial representation of an infinite reduced word in $\hat{S}_n$.  Away from crossings the wires run along $n$ positions.  A factor $s_i$ in a reduced word is depicted by a crossing of the wires in positions $i$ and $i+1$ (all indices being taken modulo $n$).  We draw all pictures with the cylinder running from left to right and with the positions numbered $1$ through $n$ from bottom to top along the front of the cylinder.  The crossing of wires $n$ and $1$, corresponding to an $s_0$, take place in the back of the cylinder.  Figure \ref{fig:sol8} shows the wiring diagram for the word $s_3s_2s_1s_0s_1s_2s_3s_2s_1s_0s_1s_2s_3s_2s_1$ with $n=4$.

Given a state word $u$ with offset $k_1$, we can draw the wiring diagram for the infinite reduced word 
\begin{displaymath}
\cdots \rho^{3k_1}(u) \rho^{2k_1}(u) \rho^{k_1}(u) u \rho^{-k_1}(u)     \cdots.
\end{displaymath}
We can also insert a rotation of a carrier word $v$ and keep pushing it through the rotations of $u$ using braid moves.  The vertex weights transform under each braid move according to Lusztig's relations.  As such, the entire sweep of $v$ from left to right transforms the vertex weights according to affine dKdV.

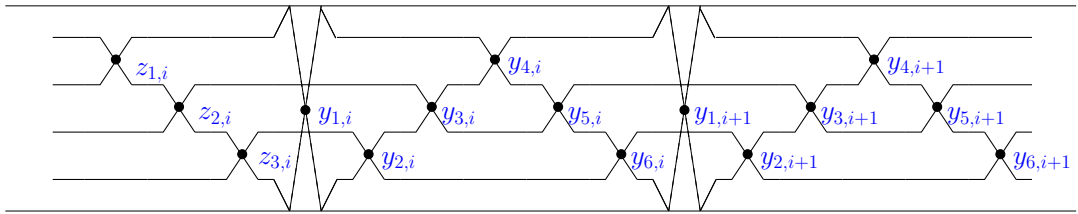
\begin{figure}[ht]
\scalebox{0.6}{\input{sol8.pstex_t}}
    \caption{A carrier $\z_i$ about to interact with the state $\y_i$.}
    \label{fig:sol8}
\end{figure}

\begin{ex} \label{ex:wire1}
 Take $n=4$, $v=s_3s_2s_1$ and $u = s_0s_1s_2s_3s_2s_1$. Then Figure \ref{fig:sol8} shows the moment when the carrier  $\z_i$ is about to interact with the state $\y_i$. The interaction is given by the following sequence of braid moves:
 $$s_3 s_2 \underline{s_1  s_0 s_1} s_2 s_3 s_2 s_1 \mapsto s_3 \underline{s_2 s_0} s_1 \underline{s_0 s_2} s_3 s_2 s_1 \mapsto s_3 s_0 \underline{s_2 s_1 s_2} s_0 s_3 s_2 s_1 \mapsto s_3 s_0 s_1 s_2 s_1 s_0  s_3 s_2 s_1.$$
 The realization of the resulting state on a cylinder is shown in Figure \ref{fig:sol9}. 
 \begin{figure}[ht]
\scalebox{0.6}{\input{sol9.pstex_t}}
    \caption{The movement right after the interaction of the carrier $\z_i$ and the state $\y_i$.}
    \label{fig:sol9}
\end{figure}
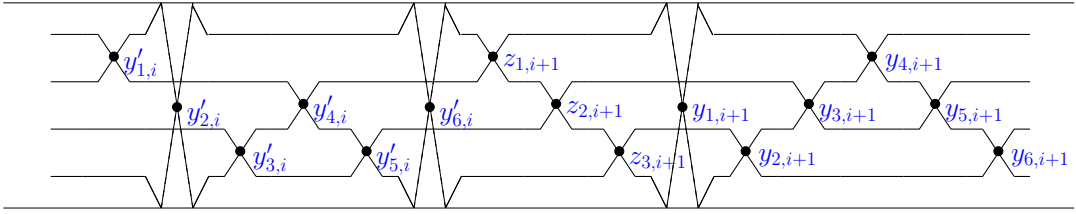
 The new parameters are related to the old ones via 
 $$y_{1,i}' = z_{1,i}, \; y_{2,i}' = \frac{y_{1,i}y_{2,i}}{z_{3,i}+y_{2,i}},\; y_{3,i}'=\frac{(z_{3,i}+y_{2,i})y_{3,i}}{z_{2,i}+y_{3,i}},\; y_{4,i}' = z_{2,i}+y_{3,i},$$
 $$y_{5,i}' = \frac{(z_{3,i}+y_{2,i})z_{2,i}}{z_{2,i}+y_{3,i}}, \; y_{6,i}' = \frac{y_{1,i}z_{3,i}}{z_{3,i}+y_{2,i}},\; z_{1,i+1} = y_{4,i},\; z_{2,i+1}=y_{5,i},\; z_{3,i+1}=y_{6,i}.$$
\end{ex}

Now, we describe a method allowing one to produce a commuting vacuum state and initial carrier. We call this method {\it {wire ansatz}}, since we use wire parameters. This method requires a consistent way to number the wires themselves, and not just the positions they occupy.  To this end, define a \emph{cut} in a reduced word (finite or infinite) to be a division of the word into two pieces.  For instance $s_1s_0|s_1s_2s_1s_0$ is a reduced word with a cut.  Given such an object, number the wires in its wiring diagram so that at the point of the cut wire $i$ is in position $i$ for all $i$ from $1$ to $n$.

Now let $u$ and $v$ be a state word and carrier word respectively.  Consider the wiring diagram for $v|u$ and assign parameters $\alpha_1, \ldots, \alpha_n$ to the $n$ wires according to the numbering method just described.  Let $\z_{-\infty}$ be a carrier giving the vertex weights for $v$ and let $\w$ be a state giving the vertex weights for $u$.

\begin{lem} \label{lemWireAnsatz}
Assume the crossing parameters in both $\z_{- \infty}$ and $\w$ satisfy $a = \beta - \alpha$, where $a$ is the parameter at a crossing of wires with parameters $\beta$ and $\alpha$, see Figure \ref{fig:sol3}. Then $F(\z_{- \infty}, \w) = (\w, \z_{- \infty}).$
Furthermore, the new parameters are again given by formulas $a = \beta - \alpha$.
\end{lem}

\begin{proof}
 It suffices to check that a single Lusztig move preserves the property of each crossing being given by $a = \beta - \alpha$ for the corresponding wire parameters. Indeed, if $a = \beta-\alpha, b = \gamma-\alpha, c=\gamma-\beta$, then $a+c=b$ and 
 $$a' = bc/(a+c) = c = \gamma-\beta, \; b' = a+c = b = \gamma -\alpha, \; c' = ab/(a+c) = a = \beta-\alpha,$$ as desired. 
\end{proof}

Consider the wiring diagram for $v|u$ and assign wire parameters $\alpha_1,\ldots, \alpha_n$.  Call the \emph{upper} wire of a crossing the one that passes from top to bottom as it crosses the other (e.g. the one labeled $\beta$ in Figure \ref{fig:sol3}).  For two wire labels $i,j$, write $i \triangleleft j$ if the wires cross with $j$ being the upper wire of the crossing.  We want all the vertex parameters in Lemma \ref{lemWireAnsatz} to be positive.  This condition holds precisely if $\alpha_i < \alpha_j$ whenever $i \triangleleft j$.

\begin{lem} \label{lem:pos}
 The parameters $\alpha_1, \ldots, \alpha_n$ can always be chosen so that $\alpha_i < \alpha_j$ whenever $i \triangleleft j$.  
\end{lem}

The proof is postponed until Section \ref{sec:prop}.

\begin{ex}
 In the Example \ref{ex:wire1} assume $\alpha_1 > \alpha_3 = \alpha_2 > \alpha_4$. Then the choice of parameters $$z_{1,i} = \alpha_1 - \alpha_4, \; z_{2,i}=\alpha_1-\alpha_3, \; z_{3,i}=\alpha_1-\alpha_2,\; y_{1,i}=\alpha_1-\alpha_4, \; y_{2,i} = \alpha_2 - \alpha_4, \; \text{etc.}$$
 results in commuting carrier  and state: $z_{j,i+1} = z_{j,i}, y_{j,i}'=y_{j,i}$. Furthermore, all parameters involved are strictly positive. 
\end{ex}

The wire ansatz is the only way we know to construct commuting pairs $\w, \z_{-\infty}$.  In fact it may be true that for typical choices of $u,v$ all such commuting pairs arise in this way.  As such, for the remainder of the paper we reformulate our setup, assuming wire weights $\alpha_1,\ldots, \alpha_n$ as independent parameters.  Once these are chosen, we can calculate $\w$ and $\z_{-\infty}$ and then define affine dKdV exactly as before.  The new setup is summarized in Table \ref{tabForm2}.

\begin{table}
\begin{tabular}{|l|l|l|}
\hline
Name & Notation & Properties \\
\hline
State word & $u = s_{i_1} \cdots s_{i_l} \in \hat{S_n}$ & $u$ a glide\\
Carrier word & $v = s_{j_1} \cdots s_{j_m} \in \hat{S_n}$ & $v$ a glide, $vu$ reduced \\
Wire weights & $\alpha_1,\ldots, \alpha_n \in \mathbb{R}$ & $i \triangleleft j \Rightarrow \alpha_i < \alpha_j $\\
\hline
\end{tabular}
\caption{The parameters for the second formulation of affine dKdV.  The entries of the vacuum $\w$ and the initial carrier $\z_{-\infty}$ are computed from $\alpha_1,\ldots, \alpha_n$ as described in Lemma \ref{lemWireAnsatz}.}
\label{tabForm2}
\end{table}

\section{Chamber variables as $\tau$-functions} \label{secBHZ}

\subsection{Chamber weights}

We recall the following concepts from \cite{BZ}. The notation and terminology is close to that of \cite{LP2}.

The affine symmetric group $\hat{S}_n$ acts on $\mathbb Z$ in the following natural way:
$$
s_i(a) = 
\begin{cases}
a+1 & \text{if $a \equiv i \mod n$;}\\
a-1 & \text{if $a \equiv i+1 \mod n$;}\\
a & \text{otherwise.}
\end{cases}
$$
This induces an action on subsets of $\mathbb Z$, including infinite ones. Let $\Z_{\leq a}$ be the set of all integers no larger than $a$. 
 Let us call a subset $S \subset \Z$ {\it $a$-nice} if it occurs as $S =
w(\Z_{\leq a})$ for some $w \in \aW$.  
Clearly a subset can be $a$-nice for at
most one $a$.  
The following lemma is easy to verify.

\begin{lem} \label{lem:nice}
 A set $S \subset \Z$ is $a$-nice for some $a$ if and only if the following three conditions hold:
 \begin{enumerate}
  \item if $b \in S$, then $b-n \in S$;
  \item $|S \cap \Z_{>0}| < \infty$;
  \item $|\Z_{<0} \setminus S| < \infty$.
 \end{enumerate}
\end{lem}

A {\it chamber weight} is an extremal weight of a fundamental
representation.  Every chamber weight is of the form $w \cdot
\omega_a$ where $w \in \aW$, and $\omega_a$ is a fundamental weight
of $\widehat{sl}(n)$.  Chamber weights in the orbit of $\omega_a$
are in bijection with $a$-nice subsets, via $w \cdot \omega_a
\leftrightarrow w(\Z_{\leq a})$.  

It shall also be convenient for us to identify each chamber weight $S$ with a lattice element $[S] = (\s_1, \ldots, \s_n) \in \Z^n$ as follows:
$$\s_i = \left\lceil\frac{\max(b \in S \mid b \equiv i \mod n)}{n}\right\rceil,$$ where $\lceil x\rceil$ denotes the smallest integer greater than or equal to $x$.
For example, $[\Z_{\leq -1}] = (0, \ldots, 0,-1)$, while $[\Z_{\leq n+1}] = (2,1, \ldots, 1)$.

Assume we are given a cylindric wiring diagram $\W$, corresponding to some biinfinite reduced word. Consider the universal cover $\bar \W$ of $\W$, which a wiring diagram with infinitely many wires, periodic under vertical shift. 
Figure \ref{fig:sol11} shows part of the universal cover of a wiring diagram. 

\begin{figure}[ht]
\scalebox{0.7}{\input{sol11.pstex_t}}
    \caption{The lifting of a wiring diagram to universal cover; wires are labeled with $\Z$; for some chambers $S$ lattice elements $[S]$ are shown.}
    \label{fig:sol11}
\end{figure}
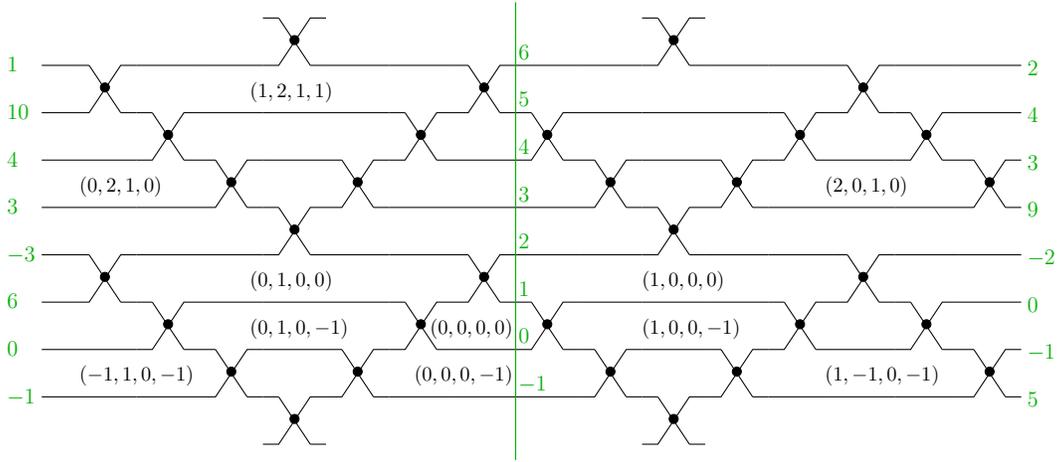

Draw a vertical line and assign to wires labels from $\Z$ according to the order in which they cross the line. The wires divide the plane into {\it {chambers}}, see \cite{BFZ}. Assign to each chamber the set $S$ of wires that pass below it.
Similarly, by keeping the labels attached to wires, we can assign a set $S$ to any chamber obtained from original $\bar \W$ by a sequence of braid moves on $\W$. The following lemma is obvious from Lemma \ref{lem:nice}.

\begin{lem}
 The resulting sets are $a$-nice, where the $a$ is determined by the horizontal level of the chamber. 
\end{lem}

From now on, we shall identify the chambers of $\bar \W$ at any stage of time evolution with the corresponding chamber weights. We label some of the chambers $S$ with their lattice elements $[S]$ in Figure \ref{fig:sol11}.  
Note that the lattice elements of two chambers obtained from each other by vertical shift differ by the vector $(1,1,1,\ldots,1)$.

The chamber labels can be used to define an important statistic of a glide, which we call its trajectory.  Let $u$ be a glide and draw the wiring diagram for $|u$ (i.e. $u$ with the cut all the way to the left).  Let $(\s_1,\ldots,\s_n)$ be the label of the chamber between wires $1$ and $2$ at the leftmost point of the wiring diagram, and let $(\s_1',\ldots,\s_n')$ be the label of the chamber between the same two wires at the rightmost point.  Say the \emph{trajectory} of $u$ is the difference 
\begin{displaymath}
t(u) = (\s_1'-\s_1,\ldots, \s_n'-\s_n).
\end{displaymath}

\begin{ex} \label{exTrajectory}
Let $n=3$ and $u=s_1s_2s_0s_2$.  The wiring diagram for $|u$ is given in Figure \ref{figTrajectory}.  Between wires $1$ and $2$ we have the labels $(1,0,0)$ at the left of the diagram and $(0,1,0)$ at the right, so $t(u) = (-1,1,0)$.  Note that the face labels, and hence the trajectory, are only defined modulo the vector $(1,1,1)$.
\end{ex}

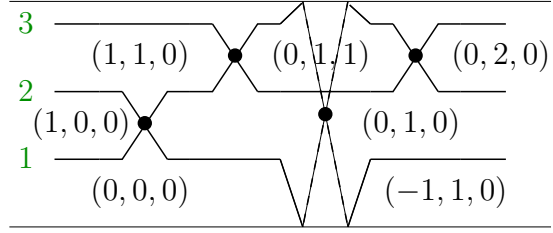
\begin{figure}
\scalebox{1.}{\input{sol24.pstex_t}}

\caption{The wiring diagram used to calculate the trajectory of $u=s_1s_2s_0s_2$}
\label{figTrajectory}
\end{figure}

There is a simple description in terms of trajectories of how the chamber labels change under the time evolution.  Let $u,v$ be glides as in the setup for affine dKdV.  Let $\tilde{v} = \rho^{-k_2}(v)$ where $k_2$ is the offset of $v$.  Let $\y_i^0=(y_{1,i}^0,\ldots, y_{l,i}^0)$ be the initial states for $i \in \mathbb{Z}$.  For $m > 0$ let $\y_i^m$ with $i \in \mathbb{Z}$ be the states after $m$ steps of the time evolution.

\begin{prop} \label{propTrajectory}
Fix $m \geq 0$ and $i \in \mathbb{Z}$.  Let $\Gamma$ be the initial network used to run affine dKdV and let $\Gamma'$ be the network after $m$ time steps.  Then there is a directed graph isomorphism of $\Gamma$ and $\Gamma'$ carrying the vertex of $\Gamma$ labeled $y_{j_0,i_0}^0$ to the vertex of $\Gamma'$ labeled $y_{j_0,i_0+i}^m$.  Let $\s$ be the label of any chamber of $\Gamma$ and let $\s'$ be the label of the associated chamber of $\Gamma'$ under the isomorphism.  Then
\begin{displaymath}
\s' = \s+it(u)-mt(\tilde{v})
\end{displaymath}  
\end{prop}

\begin{proof}
It suffices to handle the cases $i=1,m=0$ and $i=0,m=1$ as the resulting isomorphisms can be composed as needed to obtain the general case.  First let $i=1, m=0$.  Then $\Gamma = \Gamma'$ is the wiring diagram for
\begin{displaymath}
\cdots \rho^{3k_1}(u) \rho^{2k_1}(u) \rho^{k_1}(u) u \rho^{-k_1}(u) \cdots.
\end{displaymath}
The relevant isomorphism is realized by moving the cylinder to right by the length of one copy of $u$ and then rotating by $-k_1$ units (so e.g. a wire originally at position $k_1+1$ is now at position $1$.  The definition of trajectory implies that 
\begin{displaymath}
\s' = \s+t(u)
\end{displaymath}
where $\s$ is the region between wires $1$ and $2$ at the point left of $u$ (and right of $\rho^{k_1}(u)$).  The isomorphism preserves wire numbering, so $\s' = \s+t(u)$ holds for neighbors of this chamber and, by induction, all chambers.

Now suppose $i=0, m=1$.  Then $\Gamma$ is as before and $\Gamma'$ is the wiring diagrams for
\begin{displaymath}
\cdots \rho^{3k_1+k_2}(u) \rho^{2k_1+k_2}(u) \rho^{k_1+k_2}(u) \rho^{k_2}(u) \rho^{-k_1+k_2}(u) \cdots.
\end{displaymath}
The isomorphism is realized by rotating the cylinder by $k_2$ units.  To see the effect on chamber labels, consider the example of $\s$ the region left of $y_{1,i}$ in Figure \ref{fig:sol8} and $\s'$ the region left of $y_{1,i}'$ in Figure \ref{fig:sol9}.  Note that $\s'$ is also the label of the region left of $z_{1,i}$ in Figure \ref{fig:sol8}.  As $\s'$ and $\s$ are situated between the same pairs of wires to the left and right respectively of some rotation of $v$, their difference is a trajectory.  One can check the rotation is $\tilde{v}$ defined above so
\begin{displaymath}
\s' = \s - t(\tilde{v})
\end{displaymath}
as desired.
\end{proof}

\subsection{Berenstein-Hirota-Zelevinsky datum}

Assign to each lattice element $[S] \in \Z^n$ a variable $\tau_{[S]} \in \mathbb R$. We refer to those as {\it {chamber variables}}, or {\it {$\tau$-functions}}.
Assign to each wire labeled by $j \in \Z$ a parameter $\alpha_j = \alpha_i$, where $i \in [n]$ and $j \equiv i \mod n$. 
Let $e_1, \ldots, e_n$ be the standard generators of $\Z^n$.
A collection $T_{\bullet}$ of variables $\tau_{[S]}$ for all chamber weights $S$ is called {\it {Berenstein-Hirota-Zelevinsky datum}}, or {\it {BHZ datum}}, if the following relations are satisfied for any $[S] \in \Z^n$ and any distinct $i,j,k \in [n]$:
\begin{equation} \label{eqBHZ}
(\alpha_i-\alpha_j)\tau_{[S]+e_k} \tau_{[S]+e_i+e_j} + (\alpha_j-\alpha_k)\tau_{[S]+e_i} \tau_{[S]+e_j+e_k} + (\alpha_k-\alpha_i)\tau_{[S]+e_j} \tau_{[S]+e_i+e_k} = 0.
\end{equation}

For example, for $n=4$, $[S]=(0,0,0,-1)$, $i=1$, $j=2$ and $k=4$ we would have 
$$(\alpha_1-\alpha_2)\tau_{(0,0,0,0)} \tau_{(0,0,0,-1)} + (\alpha_2-\alpha_4)\tau_{(1,0,0,-1)} \tau_{(0,1,0,0)} + (\alpha_4-\alpha_1)\tau_{(0,1,0,-1)} \tau_{(1,0,0,0)} = 0.$$

If in addition we have $\tau_{[S]+e_1 + \ldots + e_n} = \tau_{[S]}$ for any $S$, we call $T_{\bullet}$ a {\it {cylindric Berenstein-Hirota-Zelevinsky datum}}.  
The following theorem allows us to reduce the problem of building $N$-soliton solutions for our system to the problem of building cylindric BHZ datum of certain form. 
\begin{figure}[ht]
\scalebox{0.6}{\input{sol12.pstex_t}}
    \caption{}
    \label{fig:sol12}
\end{figure}
\begin{thm} \label{thmBHZ}
 If $T_{\bullet}$ is a cylindric Berenstein-Hirota-Zelevinsky datum, and if crossing parameters are obtained from surrounding $\tau$-functions via 
 $$a = (\alpha_i - \alpha_j) \frac{\tau_{[S]} \tau_{[S]+e_i+e_j}}{\tau_{[S]+e_i} \tau_{[S]+e_j}},$$ cf. Figure \ref{fig:sol12}, then the same remains true after any sequence of braid moves. 
\end{thm}
Note that due to the cylindricity condition $\tau_{[S]+e_1 + \ldots + e_n} = \tau_{[S]}$ it does not matter which of the lifting of this crossing to the universal cover we choose. 

\begin{proof}
 Follows immediately from Lemma \ref{lem:YB}.
\end{proof}

\begin{rmk}
 We use the term Berenstein-Hirota-Zelevinsky datum to point out the resemblance with Berenstein-Zelevinsky datum of \cite{BZ}. 
 In the literature some other names are more common, such as {\it {Hirota bilinear difference equation}} \cite{Z}, {\it {discrete
analogue of generalized Toda equation}} and {\it {lattice KP equation}} \cite{N}, {\it { bilinear lattice KP equation}} \cite{ZFSZ}, or {\it {Hirota-Miwa equation}} \cite{LNQ}.
It goes back to the works of Miwa \cite{Mi} and Hirota \cite{H}.
Our terminology is well-suited to define cylindric Berenstein-Hirota-Zelevinsky datum, which we are not aware of appearing anywhere in the existing literature. 
\end{rmk}

\section{$N$-soliton solutions} \label{sec:nsol}

In this section, we construct soliton solutions to affine dKdV.  More precisely, we review formulas for soliton solutions to the Hirota bilinear difference equation and explain how to achieve the clyndiric condition.  The methods of the previous section then let one construct soliton solutions for any instance of affine dKdV from such a cylindric BHZ datum.

\begin{rmk}
 The Cauchy matrix approach to discrete KP type equations has been studied by Nijhoff et al \cite{NA} and by Feng, Zhao et al \cite{FZh, ZFSZ}.  We specifically cite the lecture notes \cite{N} throughout because we follow the presentation therein very closely.
\end{rmk}

\subsection{Hirota ansatz and $1$-soliton points}
The 1-soliton solutions arise from $\tau$-functions of the form
\begin{equation} \label{eq1soliton}
\tau_{[S]} = 1 + AB_1^{\s_1}B_2^{\s_2}\cdots B_n^{\s_n}.
\end{equation}
Here $A > 0$ is an arbitrary constant but $B_1,\ldots, B_n$ must be chosen carefully so that \eqref{eqBHZ} is satisfied.  Specifically, fix $b, c \in \mathbb{R}$ each distinct from $\alpha_1,\ldots, \alpha_n$ and let
\begin{displaymath}
B_j = \frac{b-\alpha_j}{c-\alpha_j}
\end{displaymath}
for $j = 1,2,\ldots, n$.

\begin{prop}[\cite{N}] \label{prop1Soliton}
For any choices of $A, b, c$, the formula \eqref{eq1soliton} gives a BHZ datum.  
\end{prop}

Now suppose $A$ is still arbitrary but $b,c$ are chosen to satisfy the relation
\begin{displaymath}
(b-\alpha_1)\cdots(b-\alpha_n) = (c-\alpha_1)\cdots(c-\alpha_n).
\end{displaymath}
In this case $B_1B_2 \cdots B_n = 1$ so
\begin{displaymath}
\tau_{[S]+e_1+\ldots+e_n} = 1 + AB_1^{\s_1+1}B_2^{\s_2+1}\cdots B_n^{\s_n+1} = 1 + AB_1^{\s_1}B_2^{\s_2}\cdots B_n^{\s_n} = \tau_{[S]}
\end{displaymath}
meaning that the datum is cylindric.

\subsection{$2$-soliton points} The $2$-soliton solutions have the general form $\tau_{[S]} = 1 + f_1 + f_2 + Zf_1f_2$ where $1+f_1$ and $1+f_2$ are both $1$-soliton solutions.  Supposing $A_1,A_2, b_1,b_2,c_1,c_2$ are all real constants we can take 
\begin{displaymath}
f_i = A_iB_{i,1}^{\s_1} \cdots B_{i,n}^{\s_n}
\end{displaymath}
where
\begin{displaymath}
B_{i,j} = \frac{b_i-\alpha_j}{c_i-\alpha_j}
\end{displaymath}
for $i=1,2$ and $j=1,2,\ldots, n$.  It remains to specify the constant $Z$.

\begin{prop}[\cite{N}] \label{prop:2}
The function $\tau_{[S]} = 1 + f_1 + f_2 + Zf_1f_2$ with $f_1,f_2$ as above and 
\begin{displaymath}
Z = \frac{(b_1-b_2)(c_1-c_2)}{(b_1-c_2)(c_1-b_2)}
\end{displaymath}
gives a BHZ datum.  
\end{prop}

Suppose in addition that 
\begin{displaymath}
(b_i-\alpha_1)\cdots(b_i-\alpha_n) = (c_i-\alpha_1)\cdots(c_i-\alpha_n)
\end{displaymath}
for $i=1,2$, implying that both component solitons are cyllindric.  It follows easily that the BHZ datum from Proposition \ref{prop:2} is cyllindric in this case.

\subsection{$N$-soliton points}  Now fix any $N \geq 1$ and let
\begin{displaymath}
B_{i,j} = \frac{b_i-\alpha_j}{c_i-\alpha_j}
\end{displaymath}
for all $i=1,\ldots, N$ and $j=1,\ldots, n$.  Let 
\begin{displaymath}
f_i = A_iB_{i,1}^{\s_1} \cdots B_{i,n}^{\s_n}
\end{displaymath}
for $i=1,\ldots, N$.  An $N$-soliton solution arises from the $\tau$-function
\begin{equation} \label{eqNSoliton}
\tau_{[S]} = \sum_{T \subseteq [N]} \prod_{\{i<j\} \subseteq T}Z_{i,j}\prod_{i \in T} f_i
\end{equation}
where 
\begin{displaymath}
Z_{i,j} = \frac{(b_i-b_j)(c_i-c_j)}{(b_i-c_j)(c_i-b_j)}.
\end{displaymath}

\begin{thm}[\cite{N}] \label{thmNSoliton}
The $\tau$-function \eqref{eqNSoliton} gives a BHZ-datum.  
\end{thm}

Here we have chosen arbitrary constant $A_i$, $b_i$, and $c_i$ for $i=1,\ldots, N$.  As in the previous cases, we can obtain a cyllindric BHZ-datum by assuming
\begin{displaymath}
(b_i-\alpha_1)\cdots(b_i-\alpha_n) = (c_i-\alpha_1)\cdots(c_i-\alpha_n)
\end{displaymath}
for $i=1,\ldots, N$.

We conclude this section by noting that the formula for $N$-solitons can also be expressed as a determinant.  The determinantal form is more common in the literature.

Consider the matrix
\begin{displaymath}
\left[\delta_{i,j} + \frac{f_i(b_j-c_j)}{b_i-c_j}\right]_{i,j=1}^N.
\end{displaymath}
Its determinant is linear in each of $f_1,\ldots, f_n$ and one can check for $T \subseteq [N]$ that the coefficient of 
\begin{displaymath}
\prod_{i \in T}f_i
\end{displaymath}
equals the determinant of the Cauchy-like matrix
\begin{displaymath}
\left[\frac{b_j-c_j}{b_i-c_j}\right]_{i,j \in T}
\end{displaymath}
which in turn equals
\begin{displaymath}
\frac{\prod_{i < j}(b_i-b_j)\prod_{i<j}(c_j-c_i)}{\prod_{i \neq j}(b_i-c_j)} = \prod_{i<j}Z_{i,j}
\end{displaymath}
where in all the products $i$ and $j$ are restricted to $T$.  Summing over $T$ we obtain
\begin{displaymath}
\left|\delta_{i,j} + \frac{f_i(b_j-c_j)}{b_i-c_j}\right|_{i,j=1}^N = \tau_{[S]}
\end{displaymath}
for $\tau_{[S]}$ as in \eqref{eqNSoliton}.

\section{Alcove model} \label{secAlcove}

\subsection{Background}

Recall the following background on the alcove model of the affine symmetric group $\hat{S}_n$. Consider the finite group $S_n$ first. The set $$\Phi^+ = \{e_i-e_j \mid 1 \leq i < j \leq n\}$$ is the set of {\it {positive roots}}. 
The ambient vector space is $V = \mathbb R^n / (1,1,\ldots,1)$, which we also identify with $\{(x_1, \ldots, x_n) \in \mathbb R^n \mid \sum_i x_i =0\}$. Denote $\alpha_{ij} = e_i-e_j$ and $\alpha_i = e_i - e_{i+1}$. The $\alpha_i$ 
are the {\it {simple roots}} of the root system. The hyperplanes $(\alpha_{ij},x) = 0$ partition the 
dual vector space $V^*$ into $n!$ {\it {chambers}}. The chamber given by $(\alpha_{ij},x) \geq 0$ for all $\alpha_{ij} \in \Phi^+$ is called the {\it {dominant chamber}} $\mathfrak C_{id}$. The extreme rays of the dominant chamber $\mu_i$, $i=1, \ldots, n-1$ are given by 
$$(\alpha_i, \mu_j) = 1 \text{ if } i=j, (\alpha_i, \mu_j) = 0 \text{ otherwise.}$$ The $\mu_i$ are the {\it {fundamental weights}}.

\begin{figure}[ht]
\scalebox{1}{\input{aff1.pstex_t}}
    \caption{The roots of the finite root system, the fundamental alcove, and the fundamental weights  for the group $\hat S_3$}
    \label{fig:aff1}
\end{figure}
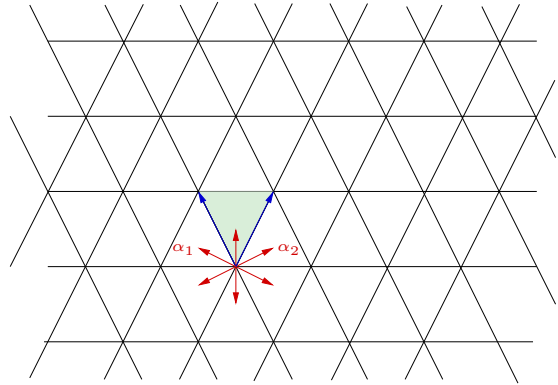

The group $S_n$ is generated by reflections $s_i$ with respect to hyperplanes $(\alpha_{i},x) = 0$. Denote $\rho = \alpha_{1n}$ the {\it {longest root}} of the root system.   If we also allow the reflection $s_0$ with respect to $(\rho,x) = 1$, 
we generate the affine symmetric group $\hat S_n$. The real roots of the affine root system correspond to the reflection hyperplanes $(\alpha_{i,j},x) = m$, where $m$ ranges over all integers. We denote the corresponding root by $\alpha_{ij}^{-m}$ if $m > 0$ and $\alpha_{ij}^{1-m}$ if $m \leq 0$ (so roots $\alpha_{ij}^m$ are indexed by $i,j,m \in \mathbb{Z}$ with $1\leq i<j\leq n$ and $m \neq 0$). 

The collection of all reflection hyperplanes
partition $V^*$ into {\it {alcoves}}. The alcove $\mathfrak A_{id}$ given by $(\alpha_{i},x) \geq 0$ and $(\rho,x) \leq 1$ is called the {\it {fundamental alcove}}. Alcoves can be identified with the elements of $\hat S_n$, by identifying $w(\mathfrak A_{id})$ with $w$ for all $w \in \hat S_n$. 
If one considers vertices of all alcoves to be elements of a ground set, and alcoves to be facets, one obtains the {\it {Coxeter complex}} of $\hat S_n$. The codimension $1$ faces of the complex are called {\it {walls}}.

Figure \ref{fig:aff1} shows the case of $n=3$. 

\subsection{Infinite reduced words and wall parameters}
Each wall separates a pair of alcoves corresponding to elements $w, w' \in \hat{S}_n$ with the property that $w' = ws_{\ell}$ for some $\ell \in \{0,1,\ldots, n-1\}$.  We label each wall with the corresponding index $\ell$.  Using this, one can identify an expression $w=s_{\ell_1}\cdots s_{\ell_k}$ with an alcove path beginning at the fundamental alcove, crossing in turn walls labeled $\ell_1,\ldots, \ell_k$, and ending at the alcove for $w$.  The expression is reduced if and only if it never crosses the same hyperplane twice.

\begin{figure}[ht]
\scalebox{1}{\input{aff2.pstex_t}}
    \caption{The alcove walk corresponding to an infinite reduced word $\dotsc (s_0 s_1) (s_2 s_0) (s_1 s_2) (s_0 s_1) \dotsc$ with a fixed choice of a cut.}
    \label{fig:aff2}
\end{figure}
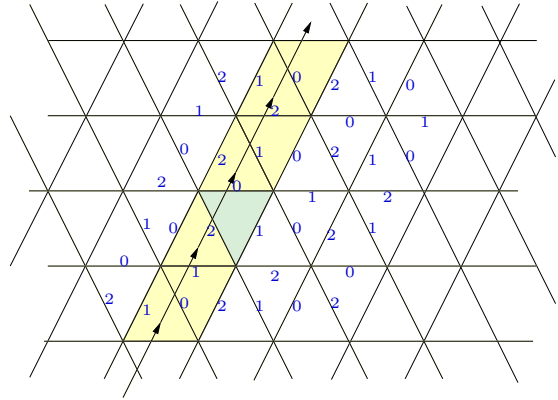
In the same way, one can identify infinite reduced words with a choice of cut with alcove walks passing through the fundamental alcove. If one performs then some braid moves, one can identify the resulting infinite reduced word with an alcove path not necessarily passing through the fundamental alcove. 

\begin{ex}
 Let $n=3$ and consider the glide $u = s_2 s_0$ of offset $1$ and the associated infinite reduced word 
 $$\dotsc \rho(u)\cdot u \cdot \rho^{-1}(u) \cdot \rho^{-2}(u) \dotsc = \dotsc (s_0 s_1) (s_2 s_0) (s_1 s_2) (s_0 s_1) \dotsc.$$
 If we choose to do the cut in the middle of $u$, we can associate to this infinite reduced word the alcove walk shown in Figure \ref{fig:aff2}.
\end{ex}

If our infinite reduced word is enriched by parameters $y_i$, $i \in \mathbb Z$, we can naturally place those parameters on walls separating the alcoves of the walk. Then once we start applying the time evolution, the wall parameters will propagate to more walls (in the case $n=3$, all walls) of the Coxeter complex.
An example of this is shown in Figure \ref{fig:aff3}, where the $y_i$-s are schematically denoted by $*$-s.

\begin{figure}[ht]
\scalebox{1}{\input{aff3.pstex_t}}
    \caption{Wall parameters propagate via carrier-state interaction.}
    \label{fig:aff3}
\end{figure}
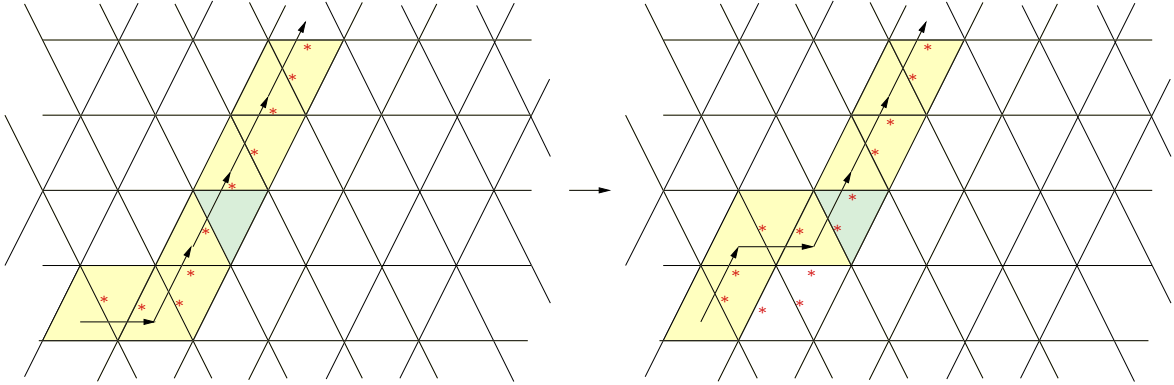

\begin{ex}
Figure \ref{figAlcoveSoliton} illustrates the parameters of the walls contained in one of the three families of hyperplanes.  Each vertical segment signifies the difference between a wall parameter and its associated value in the vacuum solution.  The specific solution plotted is a $1$-soliton.
\end{ex}


\begin{figure}
\vspace{-0.5in}
\includegraphics[width=12cm,height=12cm]{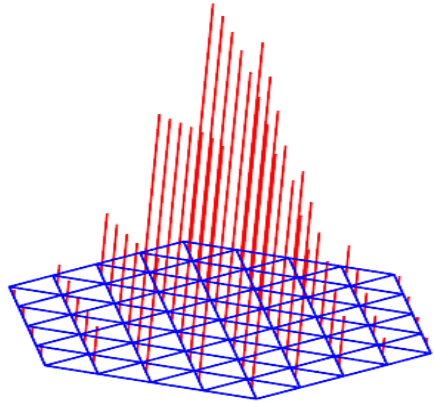}
\vspace{-1.7in}
\caption{A $1$-soliton solution depicted on top of the alcove model}
\label{figAlcoveSoliton}
\end{figure}

\subsection{Properties of the model} \label{sec:prop}
Recall the notion of trajectory $t(u)$ of a glide $u$, which was given as a difference of labels of two faces in the wiring diagram for $u$.  The face labels are weights and hence $t(u)$ is an element of the weight lattice.  Now, the weight lattice acts on $V^*$ by translations.  One can check that translation by $t(u)$ sends the fundamental alcove to the alcove $u$.  This characterization provides an alternate definition for trajectory.

\begin{rmk}
Associated to a glide $u$ are two transformations of $V^*$, one via the defining action of $\hat{S}_n$ and the other being translation by $t(u)$.  Both of these carry the fundamental alcove to the one labeled $u$.  However, they are only identical transformations if $u$ has offset $0$.
\end{rmk}

Fix wire weights $\alpha_1,\ldots, \alpha_n$.  Assume that $$\alpha_{\pi(1)} > \alpha_{\pi(2)} > \dotsc > \alpha_{\pi(n)},$$
where $\pi \in S_n$ is some permutation.  Recall the wire ansatz involves taking a wiring diagram and assigning to each crossing the parameter $\alpha_j-\alpha_i$ where $i$ is the lower wire of the crossing and $j$ is the upper wire.  As an application of the alcove model, we use it to determine when this process produces all positive parameters.

\begin{prop} \label{propPositivity}
Let $w=s_{\ell_1}\cdots s_{\ell_k}$ be a glide with a given reduced word and consider the wiring diagram for $|w$ (i.e. $w$ with the cut all the way to the left).  The crossing parameters are all positive if and only if $t(w)$ lies in $\pi(\mathfrak C_{id})$, the chamber obtained from the dominant one by applying $\pi$.
\end{prop}

\begin{proof}
Consider the alcove walk from the identity to $w$ corresponding to $s_{\ell_1}\cdots s_{\ell_k}$.  Each step corresponds to a crossing in the wiring diagram, say of wires $i$ and $j$ with $i<j$.  One can check that the hyperplane crossed during this step will have the form $(\alpha_{ij},x) = m$ for some $m \in \mathbb{Z}$ and that same $i$ and $j$.  Moreover, the hyperplane is crossed in the positive direction (that is, from the $(\alpha_{ij},x) < m$ side to the $(\alpha_{ij},x) > m$ side) if and only if $i$ is the upper wire.  

Let $t = (t_1,\ldots, t_n) = t(w)$.  First suppose the crossing parameters are all positive.  Fix any pair of distinct wires $\pi(i)$ and $\pi(j)$ with $i < j$.  For the parameter to be positive, it must be the case that $\pi(i)$ is the upper wire of each crossing.  If $\pi(i)<\pi(j)$ then $(\alpha_{\pi(i)\pi(j)},x)$ increases or stays constant along each step of the walk so $(\alpha_{\pi(i)\pi(j)},t) \geq 0$.  If $\pi(i) > \pi(j)$ then $(\alpha_{\pi(j)\pi(i)},x)$ decreases or stays constant at each step so $(\alpha_{\pi(j)\pi(i)},t) \leq 0$.  In either case we have $t_{\pi(i)} \geq t_{\pi(j)}$.  Letting $i$ and $j$ vary yields 
\begin{displaymath}
t_{\pi(1)} \geq t_{\pi(2)} \geq \ldots \geq t_{\pi(n)}
\end{displaymath} 
which exactly means $t \in \pi(\mathfrak C_{id})$.

The converse is obtained by running this same argument backwards.  The only extra point to be made is that $w$ being reduced is needed to ensure that different crossings of the same pair of wires always occur in the same direction.
\end{proof}

\begin{ex}
Suppose $\alpha_2 > \alpha_3 > \alpha_1$, so $\pi(1)=2$, $\pi(2)=3$, and $\pi(3)=1$.  Taking $w=s_1s_2s_0s_2$, one can see from Figure \ref{figTrajectory} that the crossing parameters from left to right are
\begin{displaymath}
\alpha_2-\alpha_1, \alpha_3-\alpha_1, \alpha_2-\alpha_1, \alpha_2-\alpha_3
\end{displaymath}
which are all positive.  We computed in Example \ref{exTrajectory} that $t(w) = (-1,1,0)$.  The dominant chamber consists of vectors $(x_1,\ldots, x_n)$ with $x_1 \geq x_2 \geq \ldots \geq x_n$.  Therefore $\pi^{-1}(t(w)) = (1,0,-1) \in \mathfrak C_{id}$.  This verifies that $t(w) \in \pi(\mathfrak C_{id})$.
\end{ex}

Now we are ready to prove Lemma \ref{lem:pos}.

\begin{proof}
In the setup of affine dKdV, $vu = s_{j_1}\cdots s_{j_m}s_{i_1}\cdots s_{i_l}$ is a reduced word for a glide.  By Proposition \ref{propPositivity}, there is some choice of wire weights $\alpha_1,\ldots, \alpha_n$ such that all crossing parameters in $|vu$ are positive.  More precisely, $t(vu)$ lies in some (or possibly more than one) chamber $\pi(\mathfrak C_{id})$ and we choose wire weights so that $\alpha_{\pi(1)} > \ldots > \alpha_{\pi(n)}$.  

Now, the claim of the current Lemma is that weights can be chosen so that all the crossing parameters in $v|u$ are positive.  However, $|vu$ and $v|u$ are the same networks with the wires numbered in different ways.  As such, the $\alpha_1,\ldots \alpha_n$ above can be permuted so as to work for $v|u$.  
\end{proof}

\section{Topological modes} \label{secTopological}
As we saw in Section \ref{sec:nsol}, a $1$-soliton solution for an instance of affine dKdV is expressed in terms of three parameters, $A,b,c \in \mathbb{R}$, where $A>0$ is arbitrary but $b,c$ satisfy the relation
\begin{displaymath}
(b-\alpha_1)\cdots(b-\alpha_n) = (c-\alpha_1)\cdots(c-\alpha_n)
\end{displaymath}
Adding a certain regularity condition, we get that the pair $(b,c)$ can be encoded by a directed, horizontal line segment in the plane that intersects the graph of 
\begin{displaymath}
y=(t-\alpha_1)\cdots(t-\alpha_n)
\end{displaymath}
at both its endpoints, but nowhere else.  For instance, Figure \ref{fig:sol20} gives an example of the data, which together with choices of $A_1,A_2,A_3$ determine a $3$-soliton.  In this section, we consider the effect on the soliton of the choices of components of the graph on which these segments are drawn.

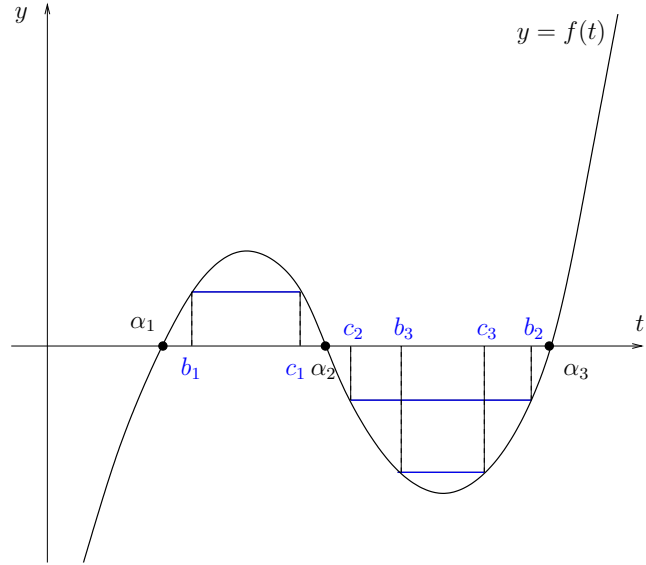
\begin{figure}[ht]
\scalebox{0.8}{\input{sol20.pstex_t}}
    \caption{Graphical description of a $3$-soliton solution.}
    \label{fig:sol20}
\end{figure}

Recall the formula for $1$-soliton solutions
\begin{displaymath}
\tau_{[S]} = 1+AB_1^{\s_1}B_2^{\s_2}\cdots B_n^{\s_n}
\end{displaymath}
where 
\begin{displaymath}
B_j = \frac{b-\alpha_j}{c-\alpha_j}.
\end{displaymath}
We impose the regularity condition $B_j>0$ for all $j$.  In other words, $b$ and $c$ should come from the same connected component of $\mathbb{R} \setminus \{\alpha_1,\ldots, \alpha_n\}$.  If $b=c$ then the $\tau$-function is constant $\tau_{[S]} = 1+A$, so to obtain non-trivial solitons we add the assumption $b \neq c$.

Now to obtain a soliton, we need the cylindricity condition  $f(b)=f(c)$ where 
\begin{displaymath}
f(t)=(t-\alpha_1)\cdots(t-\alpha_n)
\end{displaymath}
The polynomial $f$ is one-to-one on the unbounded components of $\mathbb{R} \setminus \{\alpha_1,\ldots, \alpha_n\}$ while it is generically two-to-one on the bounded components.  Therefore, we must have $\min(\{\alpha_i\}) < b < \max(\{\alpha_i\})$, and for a generic such choice of $b$ there is a unique value for $c$.  We call the component of $\mathbb{R} \setminus \{\alpha_1,\ldots, \alpha_n\}$ from which $b$ and $c$ are drawn the \emph{topological mode} of the soliton.

\subsection{Slope and speed of solitons}
Consider the vector $(\log(B_1),\ldots, \log(B_n))$ and some normal direction $(t_1,\ldots, t_n) \in \mathbb{R}^n$.  Then
\begin{displaymath}
t_1\log(B_1) + \ldots + t_n\log(B_n) = 0
\end{displaymath}
so 
\begin{displaymath}
B_1^{t_1} \cdots B_n^{t_n} = 1.
\end{displaymath}
It follows that $\tau_{[S]} = \tau_{[S]+(t_1,\ldots,t_n)}$ for all $S$.  To observe variation in the $\tau$-function, then, it suffices to travel in the $(\log(B_1),\ldots, \log(B_n))$ direction.  Hence we call this direction the \emph{slope} of the soliton.

For convenience, sort $\alpha_1,\ldots, \alpha_n$ into nondecreasing order, say
\begin{displaymath}
\alpha_{\omega(1)} \leq \alpha_{\omega(2)} \leq \ldots \leq \alpha_{\omega(n)}
\end{displaymath}
with $\omega \in S_n$.  Fix $k \in\{1,2,\ldots, n-1\}$ and consider a soliton with topological mode $(\alpha_{\omega(k)},\alpha_{\omega(k+1)})$.  Assume $b > c$ (otherwise, we can switch $b$ and $c$ which negates $\log(B_i)$ for all $i$ and hence does not affect the slope).  

\begin{lem} \label{lemSlope}
\begin{displaymath}
B_{\omega(k+1)} \leq B_{\omega(k+2)} \leq \ldots \leq B_{\omega(n)} < 1 < B_{\omega(1)} \leq B_{\omega(2)} \leq \ldots \leq B_{\omega(k)}
\end{displaymath}
\end{lem}

\begin{proof}
Suppose $b > c$ and let
\begin{displaymath}
g(t) = \frac{b-t}{c-t}
\end{displaymath}
which is increasing away from $t=c$ and which approaches $1$ in the limit in both directions.  Now $B_i=g(\alpha_i)$ so
\begin{displaymath}
\alpha_{\omega(1)} \leq \ldots \leq \alpha_{\omega(k)} < c < \alpha_{\omega(k+1)} \leq \ldots \leq \alpha_{\omega(n)}
\end{displaymath}
implies the desired result.
\end{proof}

Now fix $u$, $v$ and $\alpha_1,\ldots, \alpha_n$ as in the setup for affine dKdV (see Table \ref{tabForm2}).  Let $\tilde{v} = \rho^{-k}(v)$ where $k$ is the offset of $v$.  Let $t(u) = (t_1,\ldots, t_n)$ and $t(\tilde{v}) = (t_1',\ldots, t_n')$.  Finally, fix a $1$-soliton with $b>c$ as above.

\begin{prop} \label{propStateDirection}
\begin{displaymath}
t_1\log(B_1) + \ldots + t_n\log(B_n) < 0.
\end{displaymath}
\end{prop}

\begin{proof}
The left hand side can be rewritten
\begin{displaymath}
t_{\omega(1)} \log(B_{\omega(1)}) + \ldots + t_{\omega(n)}\log(B_{\omega(n)})
\end{displaymath}
with $\omega$ a permutation such that $\alpha_{\omega(1)} \leq \ldots \leq \alpha_{\omega(n)}$.  By Lemma \ref{lemSlope}, there exists a $k \in \{1,2,\ldots, n-1\}$ such that $\log(B_{\omega(i)})$ is positive for $i \leq k$ and negative for $i > k$.  
It is an assumption of the system that the crossing parameters of $v|u$ are positive, so the same is true for $|u$.
It follows from Proposition \ref{propPositivity} that 
\begin{displaymath}
t_{\omega(1)} \leq t_{\omega(2)} \leq \ldots \leq t_{\omega(n)}
\end{displaymath}
(note differing conventions between $\omega$ here and $\pi$ in Proposition \ref{propPositivity} which are related by $\pi(i) = \omega(n+1-i)$).
Lastly, $B_1\cdots B_n=1$ so $\log(B_1) + \ldots + \log(B_n)=0$.

Putting all the above together
\begin{align*}
t_{\omega(1)} &\log(B_{\omega(1)}) + \ldots + t_{\omega(k)}\log(B_{\omega(k)}) \\
&\leq t_{\omega(k)}( \log(B_{\omega(1)}) + \ldots + \log(B_{\omega(k)})) \\
&= t_{\omega(k)}( -\log(B_{\omega(k+1)}) - \ldots - \log(B_{\omega(n)})) \\
&< -t_{\omega(k+1)}\log(B_{\omega(k+1)}) - \ldots - t_{\omega(n)}\log(B_{\omega(n)})
\end{align*}
\end{proof}

\begin{prop} \label{propCarrierDirection}
\begin{displaymath}
t_1'\log(B_1) + \ldots + t_n'\log(B_n) < 0.
\end{displaymath}
\end{prop}

\begin{proof}
Adapting the proof of the previous proposition, it suffices to show $t'_{\omega(1)} \leq \ldots \leq t'_{\omega(n)}$.  Now the crossing parameters in $v|u$ are positive so the same is true of $v|$.  The definition of $\tilde{v}$ is such that $|\tilde{v}$ and $v|$ have a common numbering of wires.  Therefore Proposition \ref{propPositivity} can be applied to $\tilde{v}$ to get the desired result.
\end{proof}

Setting $i=0$ and $m=1$ in Proposition \ref{propTrajectory}, we see that one step of affine dKdV shifts labels in the network model by $-t(\tilde{v})$.  In the case of a $1$-soliton, this can be thought of as advancing by
\begin{displaymath}
-t_1'\log(B_1) - \ldots - t_n'\log(B_n)
\end{displaymath}
along the soliton.  Assuming the soliton is moving left to right, the effect should be the same as keeping time the same ($m=0$) but decreasing the position by some amount ($i = -p$ for some $p>0$).  The effect on labels is a shift by $-pt(u)$, so we advance in the soliton by
\begin{displaymath}
-pt_1\log(B_1) - \ldots - pt_n\log(B_n).
\end{displaymath}
Setting these two equal yields
\begin{displaymath}
p = \frac{t_1'\log(B_1) + \ldots + t_n'\log(B_n)}{t_1\log(B_1) + \ldots + t_n\log(B_n)}.
\end{displaymath}
We call $p$ the \emph{speed} of the soliton.  

\begin{rmk}
The above is simply motivation for the definition of speed.  In reality $p$ is most likely not an integer, so it does not make sense to say the soliton moves by $p$ each time step.  That said, $p$ does reflect the apparent speed as can be approximated by running the system a large number of steps.
\end{rmk}

With speed defined, it is now easy to see that it is in fact always positive.

\begin{prop}
The speed $p$ of any $1$-soliton is positive.  As such, all solitons appear to move from left to right as time increases. 
\end{prop}
\begin{proof}
If $b>c$ for $b$, $c$ the parameters of the soliton, then Propositions \ref{propStateDirection} and \ref{propCarrierDirection} show
\begin{align*}
t_1\log(B_1) + \ldots + t_n\log(B_n) &< 0 \\
t_1'\log(B_1) + \ldots + t_n'\log(B_n) &< 0 
\end{align*}
It follows that $p > 0$.  As explained above, the case $b<c$ can be obtained by interchanging $b$ and $c$ which has the effect of negating each $\log(B_i)$.  From the formula, the speed remains the same and in particular is still positive.
\end{proof}

The following example illustrates an unusual feature of our systems, namely that there are gaps between the possible speeds that solitons can have.  These jumps are observed by choosing solitons coming from different topological modes.

\begin{ex} \label{exSpeed}
Let $u = s_1s_2s_1s_0$ and $v=s_1s_0$.  Then $t(u) = (-1,0,1)$ and $t(\tilde{v}) = t(s_2s_1) = (0,0,1)$.  Let $\alpha_1=1$, $\alpha_2=3$ and $\alpha_3=4$, so $f(x)=(x-1)(x-3)(x-4)$.  We can choose $b$ from either of the intervals $(1,3)$ or $(3,4)$.  Given $b$, there is a unique $c \neq b$ in the same interval so that $f(b)=f(c)$.  The speed is then
\begin{displaymath}
p = \frac{\log(B_3)}{-\log(B_1) + \log(B_3)}
\end{displaymath}
where $B_j = \frac{b-\alpha_j}{c-\alpha_j}$.  The speed is plotted in Figure \ref{figSpeed}.  The solitons with $1<b<3$ are significantly slower than those with $3<b<4$.  There are no solitons with speed lying in an interval from around $p=0.3$ to $p=1$.
\end{ex}

\begin{figure}
\includegraphics[width=4cm,height=4cm]{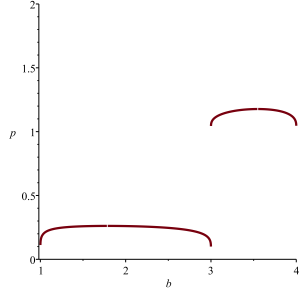}
\caption{The effect of $b$ on the speed of the resulting soliton, with the remaining parameters as given in Example \ref{exSpeed}}
\label{figSpeed}
\end{figure}

\begin{rmk}
Another meaningful statistic of a soliton is its height.  For our systems, solitons can have both positive and negative heights relative to the vacuum.  Figure \ref{figInteraction2} shows a $2$-soliton solution in which the two component solitons have heights of opposite signs.  This $2$-soliton is a solution to the same instance of affine dKdV as the one in Figure \ref{figInteraction1}.
\end{rmk}

\begin{figure}
\vspace{-0.5in}
\includegraphics[width=12cm,height=12cm]{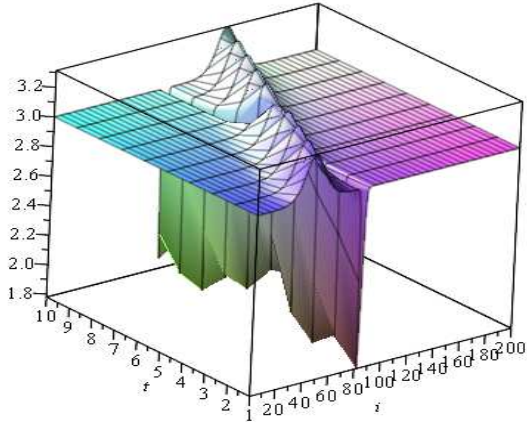}
\vspace{-1.7in}
\caption{A $2$-soliton solution in which the solitons extend in opposite directions}
\label{figInteraction2}
\end{figure}

\subsection{The symmetry}

The $N$-solitons we build are determined by the following parameters:
\begin{itemize}
 \item the roots $\alpha_i, i = 1, \ldots, n$ of the polynomial $f(t)$ -  those are  part of the model rather than merely  part of a solution;
 \item the {{ordered}} pairs $(b_i, c_i), i = 1, \ldots, n$ satisfying $f(b_i) = f(c_i)$;
 \item the parameters $A_i, i=1, \ldots, n$. 
\end{itemize}

It turns out that there is a symmetry allowing us to switch parameters $b_k$ and $c_k$ for some $k$, if we adjust the parameters $A_i$ appropriately. 
Specifically, we claim the following is true. Denote $\tau_{[S]}(b_k \leftrightarrow c_k)$ the value of the $\tau$-function we build if we swap values of parameters $b_k$ and $c_k$, which is allowed since we still have $f(b_k') = f(c_k) = f(b_k) = f(c_k')$.
Denote also $\tau_{[S]}(A_i \leftarrow A_i')$ the tau function after the substitution of the new parameters $A_i'$ instead of the old ones $A_i$.

\begin{prop} \label{prop:symm}
 We have $$\tau_{[S]}(b_k \leftrightarrow c_k) = \frac{A_k}{\prod_j B_{k,j}^{\s_j}} \cdot \tau_{[S]} \left(A_j \leftarrow \frac{A_j}{Z_{k,j}} \text{ for } j \not = k, A_k \leftarrow \frac{1}{A_k} \right).$$
\end{prop}

\begin{cor}
 The $N$-soliton solution obtained by swapping parameters $b_k$ and $c_k$ can also be obtained by keeping those parameters the same, and instead adjusting the $A_i$ parameters via $A_j \leftarrow \frac{A_j}{Z_{k,j}} \text{ for } j \not = k, A_k \leftarrow \frac{1}{A_k}$.
\end{cor}

It is easy to argue that the corollary follows from the proposition. Indeed, if $a$ is any vertex parameter in the network model then by Theorem \ref{thmBHZ} and Proposition \ref{prop:symm} we have 
$$a(b_k \leftrightarrow c_k) = (\alpha_i - \alpha_j) \frac{\tau_{[S]}(b_k \leftrightarrow c_k) \tau_{[S]+e_i+e_j}(b_k \leftrightarrow c_k)}{\tau_{[S]+e_i}(b_k \leftrightarrow c_k) \tau_{[S]+e_j}(b_k \leftrightarrow c_k)}=$$
$$= (\alpha_i - \alpha_j) \frac{\tau_{[S]}(A_j \leftarrow \frac{A_j}{Z_{k,j}} \text{ for } j \not = k, A_k \leftarrow \frac{1}{A_k}) \tau_{[S]+e_i+e_j}(A_j \leftarrow \frac{A_j}{Z_{k,j}} \text{ for } j \not = k, A_k \leftarrow \frac{1}{A_k})}
{\tau_{[S]+e_i}(A_j \leftarrow \frac{A_j}{Z_{k,j}} \text{ for } j \not = k, A_k \leftarrow \frac{1}{A_k})) \tau_{[S]+e_j}(A_j \leftarrow \frac{A_j}{Z_{k,j}} \text{ for } j \not = k, A_k \leftarrow \frac{1}{A_k})}$$
$$=a(A_j \leftarrow \frac{A_j}{Z_{k,j}} \text{ for } j \not = k, A_k \leftarrow \frac{1}{A_k})$$
since the extra factors in the numerator and the denominator cancel out: 
$$\frac{\frac{A_k}{\prod_l B_{k,l}^{\s_l}} \cdot \frac{A_k}{B_{k,i} B_{k,j} \prod_l B_{k,l}^{\s_l}}}{\frac{A_k}{B_{k,i} \prod_l B_{k,l}^{\s_l}} \cdot \frac{A_k}{B_{k,j} \prod_l B_{k,l}^{\s_l}}}=1.$$
Thus, it remains to prove the proposition.

\begin{proof}
 
It is easy to see that the swap  $b_k \leftrightarrow c_k$ has the following effect.
$$
B_{i,j}(b_k \leftrightarrow c_k) = 
\begin{cases}
1/B_{i,j} & \text{if $i=k$;}\\
B_{i,j} & \text{otherwise.}
\end{cases}
$$
$$
f_{i}(b_k \leftrightarrow c_k) = 
\begin{cases}
A_i^2/f_{i} & \text{if $i=k$;}\\
f_{i} & \text{otherwise.}
\end{cases}
$$
$$
Z_{i,j}(b_k \leftrightarrow c_k) = 
\begin{cases}
1/Z_{i,j} & \text{if $i=k$ or $j=k$;}\\
Z_{i,j} & \text{otherwise.}
\end{cases}
$$ 
Combining, we get
$$\tau_{[S]}(b_k \leftrightarrow c_k) = \sum_{T \subseteq [N], k \not \in T} \prod_{\{i<j\} \subseteq T}Z_{i,j} \cdot \prod_{i \in T} f_i + \sum_{T \subseteq [N], k \in T} 
\prod_{\{i<j\} \subseteq T, i,j \not = k}Z_{i,j} \cdot \prod_{\{j \not = k\} \subseteq T}Z_{j,k}^{-1} \cdot \frac{A_k^2}{f_k} \prod_{i \in T, i \not = k} f_i$$
$$= \frac{A_k^2}{f_k} \left[  \sum_{T \subseteq [N], k \not \in T} \prod_{\{i<j\} \subseteq T}Z_{i,j} \cdot \frac{1}{A_k^2} \prod_{i \in T \cup \{k\}} f_i + \sum_{T \subseteq [N], k \in T} 
\prod_{\{i<j\} \subseteq T/\{k\}} Z_{i,j} \cdot \prod_{\{j \not = k\} \subseteq T} Z_{j,k}^{-1} \cdot \prod_{i \in T/\{k\}} f_i   \right]$$
$$= \frac{A_k^2}{f_k} \left[  \sum_{T \subseteq [N], k \in T} \prod_{\{i<j\} \subseteq T}Z_{i,j} \cdot \prod_{\{j \not = k\} \subseteq T} Z_{j,k}^{-1} \cdot \frac{1}{A_k^2} \prod_{i \in T} f_i + \sum_{T \subseteq [N], k \not \in T} 
\prod_{\{i<j\} \subseteq T} Z_{i,j} \cdot \prod_{j \in T} Z_{j,k}^{-1} \cdot \prod_{i \in T} f_i   \right]$$ 
$$= \frac{A_k}{\prod_j B_{k,j}^{\s_j}} \cdot \tau_{[S]} \left(A_j \leftarrow \frac{A_j}{Z_{k,j}} \text{ for } j \not = k, A_k \leftarrow \frac{1}{A_k} \right).$$ 
 \end{proof}

\section{Carrier-free formulation} \label{sec:cf}
Recall the time evolution 
\begin{displaymath}
\ldots, \y_0, \y_1, \y_2 \ldots \mapsto \ldots, \y_0', \y_1', \y_2' \ldots
\end{displaymath}
 of affine dKdV is carried out by a sequence of interactions
\begin{displaymath}
(\z_i, \y_i) \mapsto (\y_i', \z_{i+1}) = F(\z_i, \y_i).
\end{displaymath}
with a carrier $\z_i$.  In principle, $\y_i'$ may depend on all the $\y_j$ with $j \leq i$ and also on the initial carrier $\z_{-\infty}$.  Oddly, for most choices $u$ and $v$ it turns out $\y'_i$ depends only on a finite window of $\y_j$ and not at all on $\z_{-\infty}$.  In these cases we obtain a new formulation for affine dKdV where 
the carrier is not needed.

\begin{ex}
Let $n=3$, $u=s_1s_2s_1s_0$, and $v = s_1s_0$.  Let $\y_i = (a_i,b_i,c_i,d_i)$, $\y_i' = (a_i',b_i',c_i',d_i')$ and $\z_i = (e_i,f_i)$.  If $F(\z_i, \y_i) = (\y_i', \z_{i+1})$, then by the formula for $F$ computed in Example \ref{exF} we have
\begin{displaymath}
\begin{array}{llll}
a_i' = \frac{a_if_i}{a_i+e_i} & b_i' = a_i+e_i & c_i' = \frac{e_if_i}{a_i+e_i} & d_i' = b_i \\
e_{i+1} = c_i & f_{i+1} = d_i
\end{array}
\end{displaymath}
This holds for all $i \in \mathbb{Z}$ so $e_i = c_{i-1}$ and $f_i = d_{i-1}$.  Therefore
\begin{displaymath}
\y_i' = \left(\frac{a_id_{i-1}}{a_i+c_{i-1}}, a_i+c_{i-1}, \frac{c_{i-1}d_{i-1}}{a_i+c_{i-1}}, b_i\right).
\end{displaymath}
We have shown $\y_i'$ depends only on $\y_{i-1}$ and $\y_i$.
\end{ex}

On the other hand, there are cases where the carrier cannot be dropped.

\begin{ex}
Let $n=3$, $u=s_1s_2$ (offset $1$), and $v=s_1s_0$ (offset $2$).  Then $F = F_{v,u}: \mathbb{R}^{2+2} \to \mathbb{R}^{2+2}$ is a weighted version of the identity
\begin{displaymath}
vu = s_1s_0s_1s_2 = s_0s_1s_0s_2 = \rho^2(u)\rho^{-1}(v)
\end{displaymath}  
which is realized by a single braid move.  Adding weights yields
\begin{displaymath}
s_1(c)s_0(d)s_1(a)s_2(b) = s_0\left(\frac{ad}{a+c}\right)s_1(a+c)s_1\left(\frac{cd}{a+c}\right)s_2(b).
\end{displaymath}
If we let $\y_i = (a_i,b_i)$, $\y_i'=(a_i',b_i')$ and $\z_i=(c_i,d_i)$ we have
\begin{displaymath}
\begin{array}{ll}
a_i' = \frac{a_id_i}{a_i+c_i} & b_i' = a_i+c_i \\
c_{i+1} = \frac{c_id_i}{a_i+c_i} & d_{i+1} = b_i 
\end{array}
\end{displaymath}
It is always possible to express $\y_i'$ in terms of $\z_{i-k}, \y_{i-k}, \y_{i-k+1}, \ldots, \y_i$.  One can check in this example that these formulas do not stabilize as $k$ increases.  This rules out the possibility of $\y_i'$ depending on only finitely many of the $\y_j$.
\end{ex}

In situations for which $\y_i'$ can be computed directly from $\y_{i-k},\y_{i-k+1}, \ldots, \y_{i}$ for some $k$, we say that the instance of affine dKdV admits a carrier free description.  In this case, we obtain a more general system as we can compute the time evolution
\begin{displaymath}
\ldots, \y_0, \y_1, \y_2, \ldots \mapsto \ldots, \y_0', \y_1', \y_2', \ldots
\end{displaymath}
for arbitrary state sequences, not just ones approaching a fixed limit (the vacuum) in both directions.  Moreover, because each new state is calculated locally, there are computational advantages such as improved numerical stability.  As a carrier free description is easier to work with, it is useful to know under what circumstances one exists.

\begin{thm} \label{thm:cfree}
Let $u$ and $v$ be the state and carrier words respectively for an instance of affine dKdV.  Assume that every pair of wires that crosses in the wiring diagram for $vu$ does so at least once in the $u$ portion of the diagram.  Then the system admits a carrier free description.
\end{thm}

Before we prove the theorem, recall some background. Let $\delta$ be the imaginary root of our affine root system. Then the real positive roots are related to the roots of the finite root system as follows:

$$
\alpha_{ij}^m = 
\begin{cases}
\alpha_{ij} + (m-1) \delta & \text{if $m > 0$;}\\
- \alpha_{ij} - m \delta & \text{if $m < 0$.}\\
\end{cases}
$$

For each $w \in \hat S_n$ the inversion set $I(w)$ is the set of roots corresponding to the hyperplanes separating the alcove of $w$ from the fundamental alcove. The following statement is well-known, see for example \cite[Section 4]{LP2}. 

\begin{lem}
 For a fixed $\alpha_{ij}$, the intersection of $I(w)$ with the set $\{\alpha_{ij}^m\}$ is one of the following:
 \begin{itemize}
  \item empty (say $M=0$);
  \item $\{\alpha_{ij}^m\}_{m=1, \ldots, M}$ for some $M>0$;
  \item $\{\alpha_{ij}^m\}_{m=M, \ldots, -1}$ for some $M<0$.
 \end{itemize}
\end{lem}

Let $m_{ij}$ be the $M$ value indicated in whichever of the three above cases applies for $\alpha_{ij}$.  This way each element $w \in \hat S_n$ corresponds to a collection of integers $\{m_{ij}\}$, one for each positive root of the finite root system. One can characterize exactly which such collections occur, see \cite[Section 4]{LP2}. More importantly for us, the following characterization of the weak Bruhat order holds, see for example \cite[Theorem 4.2]{LP2}. 

\begin{lem} \label{lemBruhat}
 If $w_1, w_2 \in \hat S_n$, then there is a reduced word for $w_2$ that starts with a reduced word for $w_1$ if and only if $I(w_1) \subseteq I(w_2)$. 
\end{lem}

Finally, the following easy lemma relates inversion sets and crossing wires, see \cite[Section 4.2]{LP2}.

\begin{lem}
 The two wires $i<j$ cross in a reduced word for $w$ if and only if $m_{ij} \not = 0$. The wire $j$ crosses wire $i$ from above if $m_{ij}>0$ and from below if $m_{ij}<0$.
\end{lem}

We are ready to prove Theorem \ref{thm:cfree}.

\begin{proof}
Let $\ldots, \y_0, \y_1, \y_2, \ldots$ be a state sequence and let $\ldots, \y_0', \y_1', \y_2', \ldots$ be the result of affine dKdV.  It suffices to provide a carrier free description of $\y_0'$, that is, a formula for $\y_0'$ in terms of $\y_{-r},\y_{-r+1}, \ldots, \y_0$ for some $r \geq 0$.  Consider the last $r$ interactions prior to the one that produces $\y_0'$.  Before the interactions, the relevant part of the infinite reduced word is 
\begin{displaymath}
\rho^{rk_1}(v)\rho^{rk_1}(u)\rho^{(r-1)k_1}(u)\cdots u
\end{displaymath}
with corresponding weights
\begin{displaymath}
\z_{-r}, \y_{-r}, \y_{-r+1}, \ldots, \y_0.
\end{displaymath}
After $r$ interactions, the reduced word has become
\begin{displaymath}
\rho^{rk_1+k_2}(u)\rho^{(r-1)k_1+k_2}(u) \cdots \rho^{k_1+k_2}(u)vu
\end{displaymath}
and the weights are 
\begin{displaymath}
\y_{-r}', \y_{-r+1}', \ldots, \y_{-1}', \z_0, \y_0.
\end{displaymath}
The next interaction produces $\y_0'$ in terms of $\z_0$ and $\y_0$, so it suffices to calculate $\z_0$ in terms of $\y_{-r},\ldots, \y_{-1}$.

Now, put a cut at the end of each of the above reduced words so that they can be viewed as alcove paths ending at the fundamental alcove.  Then all the state and carrier variables involved in these interactions naturally live at the walls crossed by these paths.  Let $w = \rho^{rk_1}(u)\rho^{(r-1)k_1}(u)\cdots u|$, namely the part of the initial word with weights $\y_{-r},\y_{-r+1}\ldots, \y_{0}$.  Taking these as given and applying braid moves, it is possible to calculate the parameter on any wall crossed by the alcove path corresponding to any other reduced word for $w$.  Meanwhile, the target $\z_0$ lies on the walls crossed by the $v$ part of the alcove path for $vu|$.  It remains to show that there is a reduced word for $w$ that ends in $vu$.  This is equivalent to saying that $(vu)^{-1}$ is less than $w^{-1}$ in the weak Bruhat order.

 The fact that the same wires cross in $vu$ as in $u$ means that $vu$ and $u$ cross the same families of hyperplanes and in the same directions. Flipping the wiring diagram backwards, we obtain the same conclusion for $(vu)^{-1}$ and $u^{-1}$.  In other words, if $\{m_{ij}\}$ and $\{m'_{ij}\}$ are the inversion numbers for $(vu)^{-1}$ and $u^{-1}$ respectively then $m_{ij}>0$
 implies $m'_{ij}>0$ and $m_{ij}<0$ implies $m'_{ij}<0$.  Each of the $r+1$ pieces of $w^{-1} = u^{-1}\rho^{k_1}(u^{-1})\cdots\rho^{rk_1}(u^{-1})$ consists of the same multiset of crossings as $u^{-1}$ does.  As such, the inversion numbers for $w^{-1}$ are $\{(r+1)m'_{ij}\}$.  Picking $r$ big enough we obtain that $|(r+1)m'_{ij}| \geq |m_{ij}|$ for all $i<j$.  It follows that $I(w^{-1}) \supseteq I((vu)^{-1})$, which combined with Lemma \ref{lemBruhat} yields the desired Bruhat relation.
\end{proof}


\medskip

\textbf{Acknowledgments.} We thank Jonathan Nimmo for suggesting a number of useful references and Thomas Lam for reading through a previous draft of this paper and providing comments.

\bibliographystyle{plain}
\bibliography{affine}{}
\end{document}

%% file: sol21.pstex_t
\begin{picture}(0,0)%
\includegraphics{sol21.pstex}%
\end{picture}%
\setlength{\unitlength}{2486sp}%
\begingroup\makeatletter\ifx\SetFigFont\undefined%
\gdef\SetFigFont#1#2#3#4#5{%
  \reset@font\fontsize{#1}{#2pt}%
  \fontfamily{#3}\fontseries{#4}\fontshape{#5}%
  \selectfont}%
\fi\endgroup%
\begin{picture}(4094,2263)(10104,-11783)
\end{picture}%

%% file: sol22.pstex_t
\begin{picture}(0,0)%
\includegraphics{sol22.pstex}%
\end{picture}%
\setlength{\unitlength}{2486sp}%
\begingroup\makeatletter\ifx\SetFigFont\undefined%
\gdef\SetFigFont#1#2#3#4#5{%
  \reset@font\fontsize{#1}{#2pt}%
  \fontfamily{#3}\fontseries{#4}\fontshape{#5}%
  \selectfont}%
\fi\endgroup%
\begin{picture}(4094,2294)(10104,-11783)
\end{picture}%

%% file: sol23.pstex_t
\begin{picture}(0,0)%
\includegraphics{sol23.pstex}%
\end{picture}%
\setlength{\unitlength}{2486sp}%
\begingroup\makeatletter\ifx\SetFigFont\undefined%
\gdef\SetFigFont#1#2#3#4#5{%
  \reset@font\fontsize{#1}{#2pt}%
  \fontfamily{#3}\fontseries{#4}\fontshape{#5}%
  \selectfont}%
\fi\endgroup%
\begin{picture}(4094,944)(10104,-11783)
\end{picture}%

%% file: aff4.pstex_t
\begin{picture}(0,0)%
\includegraphics{aff4.pstex}%
\end{picture}%
\setlength{\unitlength}{2072sp}%
\begingroup\makeatletter\ifx\SetFigFont\undefined%
\gdef\SetFigFont#1#2#3#4#5{%
  \reset@font\fontsize{#1}{#2pt}%
  \fontfamily{#3}\fontseries{#4}\fontshape{#5}%
  \selectfont}%
\fi\endgroup%
\begin{picture}(19419,4578)(6964,-7741)
\end{picture}%

%% file: sol2.pstex_t
\begin{picture}(0,0)%
\includegraphics{sol2.pstex}%
\end{picture}%
\setlength{\unitlength}{2486sp}%
\begingroup\makeatletter\ifx\SetFigFont\undefined%
\gdef\SetFigFont#1#2#3#4#5{%
  \reset@font\fontsize{#1}{#2pt}%
  \fontfamily{#3}\fontseries{#4}\fontshape{#5}%
  \selectfont}%
\fi\endgroup%
\begin{picture}(10866,3441)(1768,-11569)
\put(2701,-11401){\makebox(0,0)[lb]{\smash{{\SetFigFont{17}{20.4}{\rmdefault}{\mddefault}{\updefault}{\color[rgb]{0,0,1}$a$}%
}}}}
\put(4321,-8971){\makebox(0,0)[lb]{\smash{{\SetFigFont{17}{20.4}{\rmdefault}{\mddefault}{\updefault}{\color[rgb]{0,0,1}$b$}%
}}}}
\put(5401,-10681){\makebox(0,0)[lb]{\smash{{\SetFigFont{17}{20.4}{\rmdefault}{\mddefault}{\updefault}{\color[rgb]{0,0,1}$c$}%
}}}}
\put(9001,-8611){\makebox(0,0)[lb]{\smash{{\SetFigFont{17}{20.4}{\rmdefault}{\mddefault}{\updefault}{\color[rgb]{0,0,1}$d$}%
}}}}
\put(10531,-11041){\makebox(0,0)[lb]{\smash{{\SetFigFont{17}{20.4}{\rmdefault}{\mddefault}{\updefault}{\color[rgb]{0,0,1}$e$}%
}}}}
\put(11701,-9421){\makebox(0,0)[lb]{\smash{{\SetFigFont{17}{20.4}{\rmdefault}{\mddefault}{\updefault}{\color[rgb]{0,0,1}$f$}%
}}}}
\end{picture}%

%% file: sol3.pstex_t
\begin{picture}(0,0)%
\includegraphics{sol3.pstex}%
\end{picture}%
\setlength{\unitlength}{2486sp}%
\begingroup\makeatletter\ifx\SetFigFont\undefined%
\gdef\SetFigFont#1#2#3#4#5{%
  \reset@font\fontsize{#1}{#2pt}%
  \fontfamily{#3}\fontseries{#4}\fontshape{#5}%
  \selectfont}%
\fi\endgroup%
\begin{picture}(4458,3643)(1426,-6425)
\put(1981,-4741){\makebox(0,0)[lb]{\smash{{\SetFigFont{17}{20.4}{\rmdefault}{\mddefault}{\updefault}{\color[rgb]{0,0,0}$A$}%
}}}}
\put(3691,-3301){\makebox(0,0)[lb]{\smash{{\SetFigFont{17}{20.4}{\rmdefault}{\mddefault}{\updefault}{\color[rgb]{0,0,0}$B$}%
}}}}
\put(5131,-4741){\makebox(0,0)[lb]{\smash{{\SetFigFont{17}{20.4}{\rmdefault}{\mddefault}{\updefault}{\color[rgb]{0,0,0}$C$}%
}}}}
\put(3691,-6181){\makebox(0,0)[lb]{\smash{{\SetFigFont{17}{20.4}{\rmdefault}{\mddefault}{\updefault}{\color[rgb]{0,0,0}$D$}%
}}}}
\put(1441,-6271){\makebox(0,0)[lb]{\smash{{\SetFigFont{17}{20.4}{\rmdefault}{\mddefault}{\updefault}{\color[rgb]{0,.56,0}$\alpha$}%
}}}}
\put(1441,-3121){\makebox(0,0)[lb]{\smash{{\SetFigFont{17}{20.4}{\rmdefault}{\mddefault}{\updefault}{\color[rgb]{0,.56,0}$\beta$}%
}}}}
\put(3691,-5101){\makebox(0,0)[lb]{\smash{{\SetFigFont{17}{20.4}{\rmdefault}{\mddefault}{\updefault}{\color[rgb]{0,0,1}$a$}%
}}}}
\end{picture}%

%% file: sol4.pstex_t
\begin{picture}(0,0)%
\includegraphics{sol4.pstex}%
\end{picture}%
\setlength{\unitlength}{2486sp}%
\begingroup\makeatletter\ifx\SetFigFont\undefined%
\gdef\SetFigFont#1#2#3#4#5{%
  \reset@font\fontsize{#1}{#2pt}%
  \fontfamily{#3}\fontseries{#4}\fontshape{#5}%
  \selectfont}%
\fi\endgroup%
\begin{picture}(11388,4885)(1246,-12347)
\put(3871,-10321){\makebox(0,0)[lb]{\smash{{\SetFigFont{17}{20.4}{\rmdefault}{\mddefault}{\updefault}{\color[rgb]{0,0,0}$X$}%
}}}}
\put(10171,-9781){\makebox(0,0)[lb]{\smash{{\SetFigFont{17}{20.4}{\rmdefault}{\mddefault}{\updefault}{\color[rgb]{0,0,0}$X'$}%
}}}}
\put(1261,-11581){\makebox(0,0)[lb]{\smash{{\SetFigFont{17}{20.4}{\rmdefault}{\mddefault}{\updefault}{\color[rgb]{0,0,0}$A$}%
}}}}
\put(2341,-9601){\makebox(0,0)[lb]{\smash{{\SetFigFont{17}{20.4}{\rmdefault}{\mddefault}{\updefault}{\color[rgb]{0,0,0}$B$}%
}}}}
\put(3961,-7801){\makebox(0,0)[lb]{\smash{{\SetFigFont{17}{20.4}{\rmdefault}{\mddefault}{\updefault}{\color[rgb]{0,0,0}$C$}%
}}}}
\put(5491,-9511){\makebox(0,0)[lb]{\smash{{\SetFigFont{17}{20.4}{\rmdefault}{\mddefault}{\updefault}{\color[rgb]{0,0,0}$D$}%
}}}}
\put(6211,-11581){\makebox(0,0)[lb]{\smash{{\SetFigFont{17}{20.4}{\rmdefault}{\mddefault}{\updefault}{\color[rgb]{0,0,0}$E$}%
}}}}
\put(4051,-12031){\makebox(0,0)[lb]{\smash{{\SetFigFont{17}{20.4}{\rmdefault}{\mddefault}{\updefault}{\color[rgb]{0,0,0}$F$}%
}}}}
\put(8641,-10591){\makebox(0,0)[lb]{\smash{{\SetFigFont{17}{20.4}{\rmdefault}{\mddefault}{\updefault}{\color[rgb]{0,0,0}$A$}%
}}}}
\put(7741,-8431){\makebox(0,0)[lb]{\smash{{\SetFigFont{17}{20.4}{\rmdefault}{\mddefault}{\updefault}{\color[rgb]{0,0,0}$B$}%
}}}}
\put(10261,-7981){\makebox(0,0)[lb]{\smash{{\SetFigFont{17}{20.4}{\rmdefault}{\mddefault}{\updefault}{\color[rgb]{0,0,0}$C$}%
}}}}
\put(12601,-8521){\makebox(0,0)[lb]{\smash{{\SetFigFont{17}{20.4}{\rmdefault}{\mddefault}{\updefault}{\color[rgb]{0,0,0}$D$}%
}}}}
\put(11881,-10591){\makebox(0,0)[lb]{\smash{{\SetFigFont{17}{20.4}{\rmdefault}{\mddefault}{\updefault}{\color[rgb]{0,0,0}$E$}%
}}}}
\put(10261,-12211){\makebox(0,0)[lb]{\smash{{\SetFigFont{17}{20.4}{\rmdefault}{\mddefault}{\updefault}{\color[rgb]{0,0,0}$F$}%
}}}}
\put(2071,-11941){\makebox(0,0)[lb]{\smash{{\SetFigFont{17}{20.4}{\rmdefault}{\mddefault}{\updefault}{\color[rgb]{0,.56,0}$\alpha$}%
}}}}
\put(1621,-10771){\makebox(0,0)[lb]{\smash{{\SetFigFont{17}{20.4}{\rmdefault}{\mddefault}{\updefault}{\color[rgb]{0,.56,0}$\beta$}%
}}}}
\put(3241,-8161){\makebox(0,0)[lb]{\smash{{\SetFigFont{17}{20.4}{\rmdefault}{\mddefault}{\updefault}{\color[rgb]{0,.56,0}$\gamma$}%
}}}}
\put(4321,-8971){\makebox(0,0)[lb]{\smash{{\SetFigFont{17}{20.4}{\rmdefault}{\mddefault}{\updefault}{\color[rgb]{0,0,1}$b$}%
}}}}
\put(2701,-11401){\makebox(0,0)[lb]{\smash{{\SetFigFont{17}{20.4}{\rmdefault}{\mddefault}{\updefault}{\color[rgb]{0,0,1}$a$}%
}}}}
\put(5401,-10681){\makebox(0,0)[lb]{\smash{{\SetFigFont{17}{20.4}{\rmdefault}{\mddefault}{\updefault}{\color[rgb]{0,0,1}$c$}%
}}}}
\put(9001,-8611){\makebox(0,0)[lb]{\smash{{\SetFigFont{17}{20.4}{\rmdefault}{\mddefault}{\updefault}{\color[rgb]{0,0,1}$d$}%
}}}}
\put(10531,-11041){\makebox(0,0)[lb]{\smash{{\SetFigFont{17}{20.4}{\rmdefault}{\mddefault}{\updefault}{\color[rgb]{0,0,1}$e$}%
}}}}
\put(11701,-9421){\makebox(0,0)[lb]{\smash{{\SetFigFont{17}{20.4}{\rmdefault}{\mddefault}{\updefault}{\color[rgb]{0,0,1}$f$}%
}}}}
\put(7741,-9061){\makebox(0,0)[lb]{\smash{{\SetFigFont{17}{20.4}{\rmdefault}{\mddefault}{\updefault}{\color[rgb]{0,.56,0}$\beta$}%
}}}}
\put(8371,-7981){\makebox(0,0)[lb]{\smash{{\SetFigFont{17}{20.4}{\rmdefault}{\mddefault}{\updefault}{\color[rgb]{0,.56,0}$\gamma$}%
}}}}
\put(9541,-11851){\makebox(0,0)[lb]{\smash{{\SetFigFont{17}{20.4}{\rmdefault}{\mddefault}{\updefault}{\color[rgb]{0,.56,0}$\alpha$}%
}}}}
\end{picture}%

%% file: sol8.pstex_t
\begin{picture}(0,0)%
\includegraphics{sol8.pstex}%
\end{picture}%
\setlength{\unitlength}{2901sp}%
\begingroup\makeatletter\ifx\SetFigFont\undefined%
\gdef\SetFigFont#1#2#3#4#5{%
  \reset@font\fontsize{#1}{#2pt}%
  \fontfamily{#3}\fontseries{#4}\fontshape{#5}%
  \selectfont}%
\fi\endgroup%
\begin{picture}(15279,2969)(1114,-3008)
\put(2971,-1051){\makebox(0,0)[lb]{\smash{{\SetFigFont{17}{20.4}{\rmdefault}{\mddefault}{\updefault}{\color[rgb]{0,0,1}$z_{1,i}$}%
}}}}
\put(3871,-1681){\makebox(0,0)[lb]{\smash{{\SetFigFont{17}{20.4}{\rmdefault}{\mddefault}{\updefault}{\color[rgb]{0,0,1}$z_{2,i}$}%
}}}}
\put(5581,-1681){\makebox(0,0)[lb]{\smash{{\SetFigFont{17}{20.4}{\rmdefault}{\mddefault}{\updefault}{\color[rgb]{0,0,1}$y_{1,i}$}%
}}}}
\put(4726,-2311){\makebox(0,0)[lb]{\smash{{\SetFigFont{17}{20.4}{\rmdefault}{\mddefault}{\updefault}{\color[rgb]{0,0,1}$z_{3,i}$}%
}}}}
\put(6481,-2311){\makebox(0,0)[lb]{\smash{{\SetFigFont{17}{20.4}{\rmdefault}{\mddefault}{\updefault}{\color[rgb]{0,0,1}$y_{2,i}$}%
}}}}
\put(7336,-1681){\makebox(0,0)[lb]{\smash{{\SetFigFont{17}{20.4}{\rmdefault}{\mddefault}{\updefault}{\color[rgb]{0,0,1}$y_{3,i}$}%
}}}}
\put(8281,-961){\makebox(0,0)[lb]{\smash{{\SetFigFont{17}{20.4}{\rmdefault}{\mddefault}{\updefault}{\color[rgb]{0,0,1}$y_{4,i}$}%
}}}}
\put(9136,-1681){\makebox(0,0)[lb]{\smash{{\SetFigFont{17}{20.4}{\rmdefault}{\mddefault}{\updefault}{\color[rgb]{0,0,1}$y_{5,i}$}%
}}}}
\put(10036,-2311){\makebox(0,0)[lb]{\smash{{\SetFigFont{17}{20.4}{\rmdefault}{\mddefault}{\updefault}{\color[rgb]{0,0,1}$y_{6,i}$}%
}}}}
\put(10936,-1681){\makebox(0,0)[lb]{\smash{{\SetFigFont{17}{20.4}{\rmdefault}{\mddefault}{\updefault}{\color[rgb]{0,0,1}$y_{1,i+1}$}%
}}}}
\put(11881,-2311){\makebox(0,0)[lb]{\smash{{\SetFigFont{17}{20.4}{\rmdefault}{\mddefault}{\updefault}{\color[rgb]{0,0,1}$y_{2,i+1}$}%
}}}}
\put(13681,-961){\makebox(0,0)[lb]{\smash{{\SetFigFont{17}{20.4}{\rmdefault}{\mddefault}{\updefault}{\color[rgb]{0,0,1}$y_{4,i+1}$}%
}}}}
\put(14536,-1681){\makebox(0,0)[lb]{\smash{{\SetFigFont{17}{20.4}{\rmdefault}{\mddefault}{\updefault}{\color[rgb]{0,0,1}$y_{5,i+1}$}%
}}}}
\put(15481,-2311){\makebox(0,0)[lb]{\smash{{\SetFigFont{17}{20.4}{\rmdefault}{\mddefault}{\updefault}{\color[rgb]{0,0,1}$y_{6,i+1}$}%
}}}}
\put(12736,-1681){\makebox(0,0)[lb]{\smash{{\SetFigFont{17}{20.4}{\rmdefault}{\mddefault}{\updefault}{\color[rgb]{0,0,1}$y_{3,i+1}$}%
}}}}
\end{picture}%

%% file: sol9.pstex_t
\begin{picture}(0,0)%
\includegraphics{sol9.pstex}%
\end{picture}%
\setlength{\unitlength}{2901sp}%
\begingroup\makeatletter\ifx\SetFigFont\undefined%
\gdef\SetFigFont#1#2#3#4#5{%
  \reset@font\fontsize{#1}{#2pt}%
  \fontfamily{#3}\fontseries{#4}\fontshape{#5}%
  \selectfont}%
\fi\endgroup%
\begin{picture}(15279,2969)(1114,-3008)
\put(2836,-1006){\makebox(0,0)[lb]{\smash{{\SetFigFont{17}{20.4}{\rmdefault}{\mddefault}{\updefault}{\color[rgb]{0,0,1}$y_{1,i}'$}%
}}}}
\put(3736,-1726){\makebox(0,0)[lb]{\smash{{\SetFigFont{17}{20.4}{\rmdefault}{\mddefault}{\updefault}{\color[rgb]{0,0,1}$y_{2,i}'$}%
}}}}
\put(4681,-2356){\makebox(0,0)[lb]{\smash{{\SetFigFont{17}{20.4}{\rmdefault}{\mddefault}{\updefault}{\color[rgb]{0,0,1}$y_{3,i}'$}%
}}}}
\put(6436,-2356){\makebox(0,0)[lb]{\smash{{\SetFigFont{17}{20.4}{\rmdefault}{\mddefault}{\updefault}{\color[rgb]{0,0,1}$y_{5,i}'$}%
}}}}
\put(5536,-1681){\makebox(0,0)[lb]{\smash{{\SetFigFont{17}{20.4}{\rmdefault}{\mddefault}{\updefault}{\color[rgb]{0,0,1}$y_{4,i}'$}%
}}}}
\put(7336,-1726){\makebox(0,0)[lb]{\smash{{\SetFigFont{17}{20.4}{\rmdefault}{\mddefault}{\updefault}{\color[rgb]{0,0,1}$y_{6,i}'$}%
}}}}
\put(11881,-2311){\makebox(0,0)[lb]{\smash{{\SetFigFont{17}{20.4}{\rmdefault}{\mddefault}{\updefault}{\color[rgb]{0,0,1}$y_{2,i+1}$}%
}}}}
\put(15481,-2311){\makebox(0,0)[lb]{\smash{{\SetFigFont{17}{20.4}{\rmdefault}{\mddefault}{\updefault}{\color[rgb]{0,0,1}$y_{6,i+1}$}%
}}}}
\put(14536,-1681){\makebox(0,0)[lb]{\smash{{\SetFigFont{17}{20.4}{\rmdefault}{\mddefault}{\updefault}{\color[rgb]{0,0,1}$y_{5,i+1}$}%
}}}}
\put(13681,-961){\makebox(0,0)[lb]{\smash{{\SetFigFont{17}{20.4}{\rmdefault}{\mddefault}{\updefault}{\color[rgb]{0,0,1}$y_{4,i+1}$}%
}}}}
\put(12736,-1681){\makebox(0,0)[lb]{\smash{{\SetFigFont{17}{20.4}{\rmdefault}{\mddefault}{\updefault}{\color[rgb]{0,0,1}$y_{3,i+1}$}%
}}}}
\put(9136,-1636){\makebox(0,0)[lb]{\smash{{\SetFigFont{17}{20.4}{\rmdefault}{\mddefault}{\updefault}{\color[rgb]{0,0,1}$z_{2,i+1}$}%
}}}}
\put(8236,-961){\makebox(0,0)[lb]{\smash{{\SetFigFont{17}{20.4}{\rmdefault}{\mddefault}{\updefault}{\color[rgb]{0,0,1}$z_{1,i+1}$}%
}}}}
\put(10036,-2356){\makebox(0,0)[lb]{\smash{{\SetFigFont{17}{20.4}{\rmdefault}{\mddefault}{\updefault}{\color[rgb]{0,0,1}$z_{3,i+1}$}%
}}}}
\put(10936,-1681){\makebox(0,0)[lb]{\smash{{\SetFigFont{17}{20.4}{\rmdefault}{\mddefault}{\updefault}{\color[rgb]{0,0,1}$y_{1,i+1}$}%
}}}}
\end{picture}%

%% file: sol11.pstex_t
\begin{picture}(0,0)%
\includegraphics{sol11.pstex}%
\end{picture}%
\setlength{\unitlength}{2486sp}%
\begingroup\makeatletter\ifx\SetFigFont\undefined%
\gdef\SetFigFont#1#2#3#4#5{%
  \reset@font\fontsize{#1}{#2pt}%
  \fontfamily{#3}\fontseries{#4}\fontshape{#5}%
  \selectfont}%
\fi\endgroup%
\begin{picture}(14565,6569)(1291,-3458)
\put(8596,-2446){\makebox(0,0)[lb]{\smash{{\SetFigFont{12}{14.4}{\rmdefault}{\mddefault}{\updefault}{\color[rgb]{0,.69,0}$-1$}%
}}}}
\put(8596,-1771){\makebox(0,0)[lb]{\smash{{\SetFigFont{12}{14.4}{\rmdefault}{\mddefault}{\updefault}{\color[rgb]{0,.69,0}$0$}%
}}}}
\put(8596,-1096){\makebox(0,0)[lb]{\smash{{\SetFigFont{12}{14.4}{\rmdefault}{\mddefault}{\updefault}{\color[rgb]{0,.69,0}$1$}%
}}}}
\put(8596,-421){\makebox(0,0)[lb]{\smash{{\SetFigFont{12}{14.4}{\rmdefault}{\mddefault}{\updefault}{\color[rgb]{0,.69,0}$2$}%
}}}}
\put(8596,254){\makebox(0,0)[lb]{\smash{{\SetFigFont{12}{14.4}{\rmdefault}{\mddefault}{\updefault}{\color[rgb]{0,.69,0}$3$}%
}}}}
\put(8596,929){\makebox(0,0)[lb]{\smash{{\SetFigFont{12}{14.4}{\rmdefault}{\mddefault}{\updefault}{\color[rgb]{0,.69,0}$4$}%
}}}}
\put(8596,1604){\makebox(0,0)[lb]{\smash{{\SetFigFont{12}{14.4}{\rmdefault}{\mddefault}{\updefault}{\color[rgb]{0,.69,0}$5$}%
}}}}
\put(8596,2279){\makebox(0,0)[lb]{\smash{{\SetFigFont{12}{14.4}{\rmdefault}{\mddefault}{\updefault}{\color[rgb]{0,.69,0}$6$}%
}}}}
\put(15841,-2671){\makebox(0,0)[lb]{\smash{{\SetFigFont{12}{14.4}{\rmdefault}{\mddefault}{\updefault}{\color[rgb]{0,.69,0}$5$}%
}}}}
\put(15841,-1996){\makebox(0,0)[lb]{\smash{{\SetFigFont{12}{14.4}{\rmdefault}{\mddefault}{\updefault}{\color[rgb]{0,.69,0}$-1$}%
}}}}
\put(15841,-1321){\makebox(0,0)[lb]{\smash{{\SetFigFont{12}{14.4}{\rmdefault}{\mddefault}{\updefault}{\color[rgb]{0,.69,0}$0$}%
}}}}
\put(15841,-646){\makebox(0,0)[lb]{\smash{{\SetFigFont{12}{14.4}{\rmdefault}{\mddefault}{\updefault}{\color[rgb]{0,.69,0}$-2$}%
}}}}
\put(15841, 29){\makebox(0,0)[lb]{\smash{{\SetFigFont{12}{14.4}{\rmdefault}{\mddefault}{\updefault}{\color[rgb]{0,.69,0}$9$}%
}}}}
\put(15841,704){\makebox(0,0)[lb]{\smash{{\SetFigFont{12}{14.4}{\rmdefault}{\mddefault}{\updefault}{\color[rgb]{0,.69,0}$3$}%
}}}}
\put(15841,1379){\makebox(0,0)[lb]{\smash{{\SetFigFont{12}{14.4}{\rmdefault}{\mddefault}{\updefault}{\color[rgb]{0,.69,0}$4$}%
}}}}
\put(15841,2054){\makebox(0,0)[lb]{\smash{{\SetFigFont{12}{14.4}{\rmdefault}{\mddefault}{\updefault}{\color[rgb]{0,.69,0}$2$}%
}}}}
\put(1306,-2626){\makebox(0,0)[lb]{\smash{{\SetFigFont{12}{14.4}{\rmdefault}{\mddefault}{\updefault}{\color[rgb]{0,.69,0}$-1$}%
}}}}
\put(1306,-1951){\makebox(0,0)[lb]{\smash{{\SetFigFont{12}{14.4}{\rmdefault}{\mddefault}{\updefault}{\color[rgb]{0,.69,0}$0$}%
}}}}
\put(1306,-1276){\makebox(0,0)[lb]{\smash{{\SetFigFont{12}{14.4}{\rmdefault}{\mddefault}{\updefault}{\color[rgb]{0,.69,0}$6$}%
}}}}
\put(1306,-601){\makebox(0,0)[lb]{\smash{{\SetFigFont{12}{14.4}{\rmdefault}{\mddefault}{\updefault}{\color[rgb]{0,.69,0}$-3$}%
}}}}
\put(1306, 74){\makebox(0,0)[lb]{\smash{{\SetFigFont{12}{14.4}{\rmdefault}{\mddefault}{\updefault}{\color[rgb]{0,.69,0}$3$}%
}}}}
\put(1306,749){\makebox(0,0)[lb]{\smash{{\SetFigFont{12}{14.4}{\rmdefault}{\mddefault}{\updefault}{\color[rgb]{0,.69,0}$4$}%
}}}}
\put(1306,1424){\makebox(0,0)[lb]{\smash{{\SetFigFont{12}{14.4}{\rmdefault}{\mddefault}{\updefault}{\color[rgb]{0,.69,0}$10$}%
}}}}
\put(1306,2099){\makebox(0,0)[lb]{\smash{{\SetFigFont{12}{14.4}{\rmdefault}{\mddefault}{\updefault}{\color[rgb]{0,.69,0}$1$}%
}}}}
\put(7336,-1636){\makebox(0,0)[lb]{\smash{{\SetFigFont{11}{13.2}{\rmdefault}{\mddefault}{\updefault}{\color[rgb]{0,0,0}$(0,0,0,0)$}%
}}}}
\put(7111,-2311){\makebox(0,0)[lb]{\smash{{\SetFigFont{11}{13.2}{\rmdefault}{\mddefault}{\updefault}{\color[rgb]{0,0,0}$(0,0,0,-1)$}%
}}}}
\put(2341,-2311){\makebox(0,0)[lb]{\smash{{\SetFigFont{11}{13.2}{\rmdefault}{\mddefault}{\updefault}{\color[rgb]{0,0,0}$(-1,1,0,-1)$}%
}}}}
\put(2341,389){\makebox(0,0)[lb]{\smash{{\SetFigFont{11}{13.2}{\rmdefault}{\mddefault}{\updefault}{\color[rgb]{0,0,0}$(0,2,1,0)$}%
}}}}
\put(4771,-1636){\makebox(0,0)[lb]{\smash{{\SetFigFont{11}{13.2}{\rmdefault}{\mddefault}{\updefault}{\color[rgb]{0,0,0}$(0,1,0,-1)$}%
}}}}
\put(4771,-961){\makebox(0,0)[lb]{\smash{{\SetFigFont{11}{13.2}{\rmdefault}{\mddefault}{\updefault}{\color[rgb]{0,0,0}$(0,1,0,0)$}%
}}}}
\put(4771,1739){\makebox(0,0)[lb]{\smash{{\SetFigFont{11}{13.2}{\rmdefault}{\mddefault}{\updefault}{\color[rgb]{0,0,0}$(1,2,1,1)$}%
}}}}
\put(10351,-1636){\makebox(0,0)[lb]{\smash{{\SetFigFont{11}{13.2}{\rmdefault}{\mddefault}{\updefault}{\color[rgb]{0,0,0}$(1,0,0,-1)$}%
}}}}
\put(12961,-2311){\makebox(0,0)[lb]{\smash{{\SetFigFont{11}{13.2}{\rmdefault}{\mddefault}{\updefault}{\color[rgb]{0,0,0}$(1,-1,0,-1)$}%
}}}}
\put(12961,389){\makebox(0,0)[lb]{\smash{{\SetFigFont{11}{13.2}{\rmdefault}{\mddefault}{\updefault}{\color[rgb]{0,0,0}$(2,0,1,0)$}%
}}}}
\put(10351,-961){\makebox(0,0)[lb]{\smash{{\SetFigFont{11}{13.2}{\rmdefault}{\mddefault}{\updefault}{\color[rgb]{0,0,0}$(1,0,0,0)$}%
}}}}
\end{picture}%

%% file: sol24.pstex_t
\begin{picture}(0,0)%
\includegraphics{sol24.pstex}%
\end{picture}%
\setlength{\unitlength}{2486sp}%
\begingroup\makeatletter\ifx\SetFigFont\undefined%
\gdef\SetFigFont#1#2#3#4#5{%
  \reset@font\fontsize{#1}{#2pt}%
  \fontfamily{#3}\fontseries{#4}\fontshape{#5}%
  \selectfont}%
\fi\endgroup%
\begin{picture}(5424,2294)(3139,-10433)
\put(3241,-9826){\makebox(0,0)[lb]{\smash{{\SetFigFont{12}{14.4}{\rmdefault}{\mddefault}{\updefault}{\color[rgb]{0,.56,0}$1$}%
}}}}
\put(3241,-9151){\makebox(0,0)[lb]{\smash{{\SetFigFont{12}{14.4}{\rmdefault}{\mddefault}{\updefault}{\color[rgb]{0,.56,0}$2$}%
}}}}
\put(3241,-8476){\makebox(0,0)[lb]{\smash{{\SetFigFont{12}{14.4}{\rmdefault}{\mddefault}{\updefault}{\color[rgb]{0,.56,0}$3$}%
}}}}
\put(3961,-8791){\makebox(0,0)[lb]{\smash{{\SetFigFont{12}{14.4}{\rmdefault}{\mddefault}{\updefault}{\color[rgb]{0,0,0}$(1,1,0)$}%
}}}}
\put(3961,-10141){\makebox(0,0)[lb]{\smash{{\SetFigFont{12}{14.4}{\rmdefault}{\mddefault}{\updefault}{\color[rgb]{0,0,0}$(0,0,0)$}%
}}}}
\put(6661,-9466){\makebox(0,0)[lb]{\smash{{\SetFigFont{12}{14.4}{\rmdefault}{\mddefault}{\updefault}{\color[rgb]{0,0,0}$(0,1,0)$}%
}}}}
\put(5761,-8791){\makebox(0,0)[lb]{\smash{{\SetFigFont{12}{14.4}{\rmdefault}{\mddefault}{\updefault}{\color[rgb]{0,0,0}$(0,1,1)$}%
}}}}
\put(6886,-10141){\makebox(0,0)[lb]{\smash{{\SetFigFont{12}{14.4}{\rmdefault}{\mddefault}{\updefault}{\color[rgb]{0,0,0}$(-1,1,0)$}%
}}}}
\put(7561,-8791){\makebox(0,0)[lb]{\smash{{\SetFigFont{12}{14.4}{\rmdefault}{\mddefault}{\updefault}{\color[rgb]{0,0,0}$(0,2,0)$}%
}}}}
\put(3376,-9466){\makebox(0,0)[lb]{\smash{{\SetFigFont{12}{14.4}{\rmdefault}{\mddefault}{\updefault}{\color[rgb]{0,0,0}$(1,0,0)$}%
}}}}
\end{picture}%

%% file: sol12.pstex_t
\begin{picture}(0,0)%
\includegraphics{sol12.pstex}%
\end{picture}%
\setlength{\unitlength}{2486sp}%
\begingroup\makeatletter\ifx\SetFigFont\undefined%
\gdef\SetFigFont#1#2#3#4#5{%
  \reset@font\fontsize{#1}{#2pt}%
  \fontfamily{#3}\fontseries{#4}\fontshape{#5}%
  \selectfont}%
\fi\endgroup%
\begin{picture}(4593,3736)(1291,-6425)
\put(3691,-5101){\makebox(0,0)[lb]{\smash{{\SetFigFont{17}{20.4}{\rmdefault}{\mddefault}{\updefault}{\color[rgb]{0,0,1}$a$}%
}}}}
\put(5131,-4741){\makebox(0,0)[lb]{\smash{{\SetFigFont{17}{20.4}{\rmdefault}{\mddefault}{\updefault}{\color[rgb]{0,0,0}$\tau_{[S]+e_i}$}%
}}}}
\put(1666,-4741){\makebox(0,0)[lb]{\smash{{\SetFigFont{17}{20.4}{\rmdefault}{\mddefault}{\updefault}{\color[rgb]{0,0,0}$\tau_{[s]+e_j}$}%
}}}}
\put(1306,-3076){\makebox(0,0)[lb]{\smash{{\SetFigFont{17}{20.4}{\rmdefault}{\mddefault}{\updefault}{\color[rgb]{0,.56,0}$\alpha_i$}%
}}}}
\put(1306,-6271){\makebox(0,0)[lb]{\smash{{\SetFigFont{17}{20.4}{\rmdefault}{\mddefault}{\updefault}{\color[rgb]{0,.56,0}$\alpha_j$}%
}}}}
\put(3106,-3526){\makebox(0,0)[lb]{\smash{{\SetFigFont{17}{20.4}{\rmdefault}{\mddefault}{\updefault}{\color[rgb]{0,0,0}$\tau_{[S]+e_i+e_j}$}%
}}}}
\put(3556,-5866){\makebox(0,0)[lb]{\smash{{\SetFigFont{17}{20.4}{\rmdefault}{\mddefault}{\updefault}{\color[rgb]{0,0,0}$\tau_{[S]}$}%
}}}}
\end{picture}%

%% file: aff1.pstex_t
\begin{picture}(0,0)%
\includegraphics{aff1.pstex}%
\end{picture}%
\setlength{\unitlength}{2072sp}%
\begingroup\makeatletter\ifx\SetFigFont\undefined%
\gdef\SetFigFont#1#2#3#4#5{%
  \reset@font\fontsize{#1}{#2pt}%
  \fontfamily{#3}\fontseries{#4}\fontshape{#5}%
  \selectfont}%
\fi\endgroup%
\begin{picture}(6549,4587)(7864,-7768)
\put(11071,-6181){\makebox(0,0)[lb]{\smash{{\SetFigFont{6}{7.2}{\rmdefault}{\mddefault}{\updefault}{\color[rgb]{.82,0,0}$\alpha_2$}%
}}}}
\put(9811,-6181){\makebox(0,0)[lb]{\smash{{\SetFigFont{6}{7.2}{\rmdefault}{\mddefault}{\updefault}{\color[rgb]{.82,0,0}$\alpha_1$}%
}}}}
\end{picture}%

%% file: aff2.pstex_t
\begin{picture}(0,0)%
\includegraphics{aff2.pstex}%
\end{picture}%
\setlength{\unitlength}{2072sp}%
\begingroup\makeatletter\ifx\SetFigFont\undefined%
\gdef\SetFigFont#1#2#3#4#5{%
  \reset@font\fontsize{#1}{#2pt}%
  \fontfamily{#3}\fontseries{#4}\fontshape{#5}%
  \selectfont}%
\fi\endgroup%
\begin{picture}(6549,4767)(7864,-7948)
\put(10216,-6001){\makebox(0,0)[lb]{\smash{{\SetFigFont{6}{7.2}{\rmdefault}{\mddefault}{\updefault}{\color[rgb]{0,0,.82}$2$}%
}}}}
\put(10801,-6001){\makebox(0,0)[lb]{\smash{{\SetFigFont{6}{7.2}{\rmdefault}{\mddefault}{\updefault}{\color[rgb]{0,0,.82}$1$}%
}}}}
\put(10531,-5461){\makebox(0,0)[lb]{\smash{{\SetFigFont{6}{7.2}{\rmdefault}{\mddefault}{\updefault}{\color[rgb]{0,0,.82}$0$}%
}}}}
\put(10801,-5056){\makebox(0,0)[lb]{\smash{{\SetFigFont{6}{7.2}{\rmdefault}{\mddefault}{\updefault}{\color[rgb]{0,0,.82}$1$}%
}}}}
\put(11251,-5101){\makebox(0,0)[lb]{\smash{{\SetFigFont{6}{7.2}{\rmdefault}{\mddefault}{\updefault}{\color[rgb]{0,0,.82}$0$}%
}}}}
\put(11431,-5596){\makebox(0,0)[lb]{\smash{{\SetFigFont{6}{7.2}{\rmdefault}{\mddefault}{\updefault}{\color[rgb]{0,0,.82}$1$}%
}}}}
\put(11251,-5956){\makebox(0,0)[lb]{\smash{{\SetFigFont{6}{7.2}{\rmdefault}{\mddefault}{\updefault}{\color[rgb]{0,0,.82}$0$}%
}}}}
\put(10351,-5146){\makebox(0,0)[lb]{\smash{{\SetFigFont{6}{7.2}{\rmdefault}{\mddefault}{\updefault}{\color[rgb]{0,0,.82}$2$}%
}}}}
\put(9901,-5011){\makebox(0,0)[lb]{\smash{{\SetFigFont{6}{7.2}{\rmdefault}{\mddefault}{\updefault}{\color[rgb]{0,0,.82}$0$}%
}}}}
\put(9631,-5416){\makebox(0,0)[lb]{\smash{{\SetFigFont{6}{7.2}{\rmdefault}{\mddefault}{\updefault}{\color[rgb]{0,0,.82}$2$}%
}}}}
\put(9766,-5956){\makebox(0,0)[lb]{\smash{{\SetFigFont{6}{7.2}{\rmdefault}{\mddefault}{\updefault}{\color[rgb]{0,0,.82}$0$}%
}}}}
\put(10981,-6541){\makebox(0,0)[lb]{\smash{{\SetFigFont{6}{7.2}{\rmdefault}{\mddefault}{\updefault}{\color[rgb]{0,0,.82}$2$}%
}}}}
\put(10801,-6901){\makebox(0,0)[lb]{\smash{{\SetFigFont{6}{7.2}{\rmdefault}{\mddefault}{\updefault}{\color[rgb]{0,0,.82}$1$}%
}}}}
\put(10351,-6901){\makebox(0,0)[lb]{\smash{{\SetFigFont{6}{7.2}{\rmdefault}{\mddefault}{\updefault}{\color[rgb]{0,0,.82}$2$}%
}}}}
\put(10036,-6496){\makebox(0,0)[lb]{\smash{{\SetFigFont{6}{7.2}{\rmdefault}{\mddefault}{\updefault}{\color[rgb]{0,0,.82}$1$}%
}}}}
\put(10981,-4561){\makebox(0,0)[lb]{\smash{{\SetFigFont{6}{7.2}{\rmdefault}{\mddefault}{\updefault}{\color[rgb]{0,0,.82}$2$}%
}}}}
\put(10801,-4201){\makebox(0,0)[lb]{\smash{{\SetFigFont{6}{7.2}{\rmdefault}{\mddefault}{\updefault}{\color[rgb]{0,0,.82}$1$}%
}}}}
\put(10351,-4156){\makebox(0,0)[lb]{\smash{{\SetFigFont{6}{7.2}{\rmdefault}{\mddefault}{\updefault}{\color[rgb]{0,0,.82}$2$}%
}}}}
\put(10081,-4561){\makebox(0,0)[lb]{\smash{{\SetFigFont{6}{7.2}{\rmdefault}{\mddefault}{\updefault}{\color[rgb]{0,0,.82}$1$}%
}}}}
\put(11251,-4156){\makebox(0,0)[lb]{\smash{{\SetFigFont{6}{7.2}{\rmdefault}{\mddefault}{\updefault}{\color[rgb]{0,0,.82}$0$}%
}}}}
\put(11701,-4246){\makebox(0,0)[lb]{\smash{{\SetFigFont{6}{7.2}{\rmdefault}{\mddefault}{\updefault}{\color[rgb]{0,0,.82}$2$}%
}}}}
\put(11881,-4696){\makebox(0,0)[lb]{\smash{{\SetFigFont{6}{7.2}{\rmdefault}{\mddefault}{\updefault}{\color[rgb]{0,0,.82}$0$}%
}}}}
\put(11701,-5056){\makebox(0,0)[lb]{\smash{{\SetFigFont{6}{7.2}{\rmdefault}{\mddefault}{\updefault}{\color[rgb]{0,0,.82}$2$}%
}}}}
\put(9451,-5911){\makebox(0,0)[lb]{\smash{{\SetFigFont{6}{7.2}{\rmdefault}{\mddefault}{\updefault}{\color[rgb]{0,0,.82}$1$}%
}}}}
\put(9181,-6361){\makebox(0,0)[lb]{\smash{{\SetFigFont{6}{7.2}{\rmdefault}{\mddefault}{\updefault}{\color[rgb]{0,0,.82}$0$}%
}}}}
\put(9451,-6946){\makebox(0,0)[lb]{\smash{{\SetFigFont{6}{7.2}{\rmdefault}{\mddefault}{\updefault}{\color[rgb]{0,0,.82}$1$}%
}}}}
\put(9901,-6856){\makebox(0,0)[lb]{\smash{{\SetFigFont{6}{7.2}{\rmdefault}{\mddefault}{\updefault}{\color[rgb]{0,0,.82}$0$}%
}}}}
\put(9001,-6811){\makebox(0,0)[lb]{\smash{{\SetFigFont{6}{7.2}{\rmdefault}{\mddefault}{\updefault}{\color[rgb]{0,0,.82}$2$}%
}}}}
\put(12151,-4156){\makebox(0,0)[lb]{\smash{{\SetFigFont{6}{7.2}{\rmdefault}{\mddefault}{\updefault}{\color[rgb]{0,0,.82}$1$}%
}}}}
\put(12601,-4246){\makebox(0,0)[lb]{\smash{{\SetFigFont{6}{7.2}{\rmdefault}{\mddefault}{\updefault}{\color[rgb]{0,0,.82}$0$}%
}}}}
\put(12781,-4696){\makebox(0,0)[lb]{\smash{{\SetFigFont{6}{7.2}{\rmdefault}{\mddefault}{\updefault}{\color[rgb]{0,0,.82}$1$}%
}}}}
\put(12601,-5101){\makebox(0,0)[lb]{\smash{{\SetFigFont{6}{7.2}{\rmdefault}{\mddefault}{\updefault}{\color[rgb]{0,0,.82}$0$}%
}}}}
\put(12151,-5146){\makebox(0,0)[lb]{\smash{{\SetFigFont{6}{7.2}{\rmdefault}{\mddefault}{\updefault}{\color[rgb]{0,0,.82}$1$}%
}}}}
\put(12331,-5596){\makebox(0,0)[lb]{\smash{{\SetFigFont{6}{7.2}{\rmdefault}{\mddefault}{\updefault}{\color[rgb]{0,0,.82}$2$}%
}}}}
\put(12151,-5956){\makebox(0,0)[lb]{\smash{{\SetFigFont{6}{7.2}{\rmdefault}{\mddefault}{\updefault}{\color[rgb]{0,0,.82}$1$}%
}}}}
\put(11656,-6046){\makebox(0,0)[lb]{\smash{{\SetFigFont{6}{7.2}{\rmdefault}{\mddefault}{\updefault}{\color[rgb]{0,0,.82}$2$}%
}}}}
\put(11881,-6496){\makebox(0,0)[lb]{\smash{{\SetFigFont{6}{7.2}{\rmdefault}{\mddefault}{\updefault}{\color[rgb]{0,0,.82}$0$}%
}}}}
\put(11701,-6856){\makebox(0,0)[lb]{\smash{{\SetFigFont{6}{7.2}{\rmdefault}{\mddefault}{\updefault}{\color[rgb]{0,0,.82}$2$}%
}}}}
\put(11251,-6901){\makebox(0,0)[lb]{\smash{{\SetFigFont{6}{7.2}{\rmdefault}{\mddefault}{\updefault}{\color[rgb]{0,0,.82}$0$}%
}}}}
\end{picture}%

%% file: aff3.pstex_t
\begin{picture}(0,0)%
\includegraphics{aff3.pstex}%
\end{picture}%
\setlength{\unitlength}{2072sp}%
\begingroup\makeatletter\ifx\SetFigFont\undefined%
\gdef\SetFigFont#1#2#3#4#5{%
  \reset@font\fontsize{#1}{#2pt}%
  \fontfamily{#3}\fontseries{#4}\fontshape{#5}%
  \selectfont}%
\fi\endgroup%
\begin{picture}(13974,4581)(7864,-7762)
\put(9901,-6856){\makebox(0,0)[lb]{\smash{{\SetFigFont{6}{7.2}{\rmdefault}{\mddefault}{\updefault}{\color[rgb]{.82,0,0}$*$}%
}}}}
\put(9001,-6811){\makebox(0,0)[lb]{\smash{{\SetFigFont{6}{7.2}{\rmdefault}{\mddefault}{\updefault}{\color[rgb]{.82,0,0}$*$}%
}}}}
\put(10036,-6496){\makebox(0,0)[lb]{\smash{{\SetFigFont{6}{7.2}{\rmdefault}{\mddefault}{\updefault}{\color[rgb]{.82,0,0}$*$}%
}}}}
\put(10216,-6001){\makebox(0,0)[lb]{\smash{{\SetFigFont{6}{7.2}{\rmdefault}{\mddefault}{\updefault}{\color[rgb]{.82,0,0}$*$}%
}}}}
\put(10531,-5461){\makebox(0,0)[lb]{\smash{{\SetFigFont{6}{7.2}{\rmdefault}{\mddefault}{\updefault}{\color[rgb]{.82,0,0}$*$}%
}}}}
\put(10801,-5056){\makebox(0,0)[lb]{\smash{{\SetFigFont{6}{7.2}{\rmdefault}{\mddefault}{\updefault}{\color[rgb]{.82,0,0}$*$}%
}}}}
\put(11251,-4156){\makebox(0,0)[lb]{\smash{{\SetFigFont{6}{7.2}{\rmdefault}{\mddefault}{\updefault}{\color[rgb]{.82,0,0}$*$}%
}}}}
\put(16876,-6946){\makebox(0,0)[lb]{\smash{{\SetFigFont{6}{7.2}{\rmdefault}{\mddefault}{\updefault}{\color[rgb]{.82,0,0}$*$}%
}}}}
\put(16426,-6811){\makebox(0,0)[lb]{\smash{{\SetFigFont{6}{7.2}{\rmdefault}{\mddefault}{\updefault}{\color[rgb]{.82,0,0}$*$}%
}}}}
\put(17326,-6856){\makebox(0,0)[lb]{\smash{{\SetFigFont{6}{7.2}{\rmdefault}{\mddefault}{\updefault}{\color[rgb]{.82,0,0}$*$}%
}}}}
\put(17461,-6496){\makebox(0,0)[lb]{\smash{{\SetFigFont{6}{7.2}{\rmdefault}{\mddefault}{\updefault}{\color[rgb]{.82,0,0}$*$}%
}}}}
\put(18226,-5056){\makebox(0,0)[lb]{\smash{{\SetFigFont{6}{7.2}{\rmdefault}{\mddefault}{\updefault}{\color[rgb]{.82,0,0}$*$}%
}}}}
\put(18676,-4156){\makebox(0,0)[lb]{\smash{{\SetFigFont{6}{7.2}{\rmdefault}{\mddefault}{\updefault}{\color[rgb]{.82,0,0}$*$}%
}}}}
\put(9451,-6901){\makebox(0,0)[lb]{\smash{{\SetFigFont{6}{7.2}{\rmdefault}{\mddefault}{\updefault}{\color[rgb]{.82,0,0}$*$}%
}}}}
\put(11026,-4561){\makebox(0,0)[lb]{\smash{{\SetFigFont{6}{7.2}{\rmdefault}{\mddefault}{\updefault}{\color[rgb]{.82,0,0}$*$}%
}}}}
\put(11431,-3796){\makebox(0,0)[lb]{\smash{{\SetFigFont{6}{7.2}{\rmdefault}{\mddefault}{\updefault}{\color[rgb]{.82,0,0}$*$}%
}}}}
\put(16561,-6496){\makebox(0,0)[lb]{\smash{{\SetFigFont{6}{7.2}{\rmdefault}{\mddefault}{\updefault}{\color[rgb]{.82,0,0}$*$}%
}}}}
\put(16876,-5956){\makebox(0,0)[lb]{\smash{{\SetFigFont{6}{7.2}{\rmdefault}{\mddefault}{\updefault}{\color[rgb]{.82,0,0}$*$}%
}}}}
\put(17326,-6001){\makebox(0,0)[lb]{\smash{{\SetFigFont{6}{7.2}{\rmdefault}{\mddefault}{\updefault}{\color[rgb]{.82,0,0}$*$}%
}}}}
\put(17776,-5956){\makebox(0,0)[lb]{\smash{{\SetFigFont{6}{7.2}{\rmdefault}{\mddefault}{\updefault}{\color[rgb]{.82,0,0}$*$}%
}}}}
\put(17956,-5596){\makebox(0,0)[lb]{\smash{{\SetFigFont{6}{7.2}{\rmdefault}{\mddefault}{\updefault}{\color[rgb]{.82,0,0}$*$}%
}}}}
\put(18406,-4696){\makebox(0,0)[lb]{\smash{{\SetFigFont{6}{7.2}{\rmdefault}{\mddefault}{\updefault}{\color[rgb]{.82,0,0}$*$}%
}}}}
\put(18856,-3796){\makebox(0,0)[lb]{\smash{{\SetFigFont{6}{7.2}{\rmdefault}{\mddefault}{\updefault}{\color[rgb]{.82,0,0}$*$}%
}}}}
\end{picture}%

%% file: sol20.pstex_t
\begin{picture}(0,0)%
\includegraphics{sol20.pstex}%
\end{picture}%
\setlength{\unitlength}{2486sp}%
\begingroup\makeatletter\ifx\SetFigFont\undefined%
\gdef\SetFigFont#1#2#3#4#5{%
  \reset@font\fontsize{#1}{#2pt}%
  \fontfamily{#3}\fontseries{#4}\fontshape{#5}%
  \selectfont}%
\fi\endgroup%
\begin{picture}(7899,7084)(3589,-9758)
\put(5086,-6811){\makebox(0,0)[lb]{\smash{{\SetFigFont{12}{14.4}{\rmdefault}{\mddefault}{\updefault}{\color[rgb]{0,0,0}$\alpha_1$}%
}}}}
\put(9991,-6901){\makebox(0,0)[lb]{\smash{{\SetFigFont{12}{14.4}{\rmdefault}{\mddefault}{\updefault}{\color[rgb]{0,0,1}$b_2$}%
}}}}
\put(7741,-6901){\makebox(0,0)[lb]{\smash{{\SetFigFont{12}{14.4}{\rmdefault}{\mddefault}{\updefault}{\color[rgb]{0,0,1}$c_2$}%
}}}}
\put(8371,-6901){\makebox(0,0)[lb]{\smash{{\SetFigFont{12}{14.4}{\rmdefault}{\mddefault}{\updefault}{\color[rgb]{0,0,1}$b_3$}%
}}}}
\put(9406,-6901){\makebox(0,0)[lb]{\smash{{\SetFigFont{12}{14.4}{\rmdefault}{\mddefault}{\updefault}{\color[rgb]{0,0,1}$c_3$}%
}}}}
\put(5716,-7396){\makebox(0,0)[lb]{\smash{{\SetFigFont{12}{14.4}{\rmdefault}{\mddefault}{\updefault}{\color[rgb]{0,0,1}$b_1$}%
}}}}
\put(7021,-7396){\makebox(0,0)[lb]{\smash{{\SetFigFont{12}{14.4}{\rmdefault}{\mddefault}{\updefault}{\color[rgb]{0,0,1}$c_1$}%
}}}}
\put(7336,-7396){\makebox(0,0)[lb]{\smash{{\SetFigFont{12}{14.4}{\rmdefault}{\mddefault}{\updefault}{\color[rgb]{0,0,0}$\alpha_2$}%
}}}}
\put(10486,-7396){\makebox(0,0)[lb]{\smash{{\SetFigFont{12}{14.4}{\rmdefault}{\mddefault}{\updefault}{\color[rgb]{0,0,0}$\alpha_3$}%
}}}}
\put(9901,-3211){\makebox(0,0)[lb]{\smash{{\SetFigFont{12}{14.4}{\rmdefault}{\mddefault}{\updefault}{\color[rgb]{0,0,0}$y=f(t)$}%
}}}}
\put(11386,-6856){\makebox(0,0)[lb]{\smash{{\SetFigFont{12}{14.4}{\rmdefault}{\mddefault}{\updefault}{\color[rgb]{0,0,0}$t$}%
}}}}
\put(3646,-2941){\makebox(0,0)[lb]{\smash{{\SetFigFont{12}{14.4}{\rmdefault}{\mddefault}{\updefault}{\color[rgb]{0,0,0}$y$}%
}}}}
\end{picture}%